\documentclass[reqno,11pt]{amsart}
\usepackage{amsmath,amssymb,latexsym,soul,cite,mathrsfs}
\usepackage{color,enumitem,graphicx}
\usepackage[colorlinks=true,urlcolor=blue,
citecolor=red,linkcolor=blue,linktocpage,pdfpagelabels,
bookmarksnumbered,bookmarksopen]{hyperref}
\usepackage[english]{babel}
\usepackage[left=2.6cm,right=2.6cm,top=2.9cm,bottom=2.9cm]{geometry}
\usepackage[hyperpageref]{backref}
\usepackage{lipsum}
\usepackage{enumerate}

\usepackage{mathtools}

\usepackage{tcolorbox}

\usepackage{tikz}
\usetikzlibrary{shapes.geometric, arrows}

\newtheorem{theorem}{Theorem}[section]
\newtheorem{lemma}[theorem]{Lemma}
\newtheorem{question}[theorem]{Question}
\newtheorem{example}[theorem]{Example}
\newtheorem{corollary}[theorem]{Corollary}
\newtheorem{proposition}[theorem]{Proposition}
\newtheorem{remark}[theorem]{Remark}
\newtheorem{definition}[theorem]{Definition}

\newcommand{\conv}[1]{\xrightarrow{\,#1\,}}

\newcommand{\innerthmname}{}% initialize

\theoremstyle{definition}

\makeatletter
\def\namedlabel#1#2{\begingroup
	#2%
	\def\@currentlabel{#2}%
	\phantomsection\label{#1}\endgroup
}
\makeatother

\newcommand{\Sp}{\mathbb{S}}
\newcommand{\ep}{\varepsilon}

\newcommand{\s}{\hspace{7pt}}
\newcommand{\lp}{\langle}
\newcommand{\rp}{\rangle}
\newcommand{\R}{\mathbb{R}}
\newcommand{\B}{\mathcal{B}}
\newcommand{\J}{\mathcal{J}}
\newcommand{\vsp}{\vspace{2.5pt}}

\DeclareMathOperator{\supp}{supp}

\DeclareMathOperator{\Ric}{Ric}

\usepackage{esint}

\title[Asymptotics as $s \to 0^+$ of the fractional perimeter on Riemannian manifolds.]{Asymptotics as $s \to 0^+$ of the fractional perimeter on Riemannian manifolds}

\author[M. Caselli]{Michele Caselli}
\author[L. Gennaioli]{Luca Gennaioli}
\address[M. Caselli]{
	Scuola Normale Superiore
	\newline\indent 
	Piazza dei Cavalieri 7, 56126 Pisa, Italy}
\email{\href{mailto:michele.caselli@sns.it}{michele.caselli@sns.it}}

\address[L. Gennaioli]{Scuola Internazionale Superiore di Studi Avanzati
	\newline\indent 
	Via Bonomea, 265, 34136 Trieste, Italy}
\email{\href{mailto:luca.gennaioli@sissa.it}{luca.gennaioli@sissa.it}}

\begin{document}
	
	\begin{abstract}
		In this work, we study the asymptotics of the fractional Laplacian as $s\to 0^+$ on any complete Riemannian manifold $(M,g)$, both of finite and infinite volume. Surprisingly enough, when $M$ is not stochastically complete, this asymptotics is related to the existence of bounded harmonic functions on $M$. 

  \vsp
  As a corollary, we can find the asymptotics of the fractional $s$-perimeter on (essentially) every complete manifold, generalizing both the existing results \cite{dpfv} for $\R^n$ and \cite{CaCiLaPa21a} for the Gaussian space. In doing so, from many sets $E\subset M$, we are able to produce a bounded harmonic function associated with $E$, which, in general, can be non-constant.

	\end{abstract}
	
	\maketitle
  \tableofcontents
	
	\section{Introduction}\label{section1}
This work deals with the fractional Laplacian on general complete Riemannian manifolds. Given a set $E\subset M$, our work is based on the study of the following quantity
 \begin{equation}\label{inttheta1}
         \theta_E(p):=\lim_{s\to 0^+} \int_{E \setminus B_R(p)} \mathcal{K}_s(x,p) d\mu(x)  \,, 
     \end{equation}
 where 
 \begin{equation}\label{singkerdef}
\mathcal{K}_s(x,y) :=  
 \frac{1}{|\Gamma(-s/2)|}\int_{0}^\infty H_M(x,y,t)\,\frac{dt}{t^{1+s/2}}
 \end{equation}
 and $H_M (x,y,t): M \times M \times (0,\infty)$ is the heat kernel of $M$, that is the minimal\footnote{Here, minimal means the following: if $v : (0, \infty) \times M \to \R $ is another function with $\partial_t v -\Delta_g v = 0 $ on $M$ and $ v(t,\cdot) \to \delta_{\{y\}}$ as $t \to 0^+$, then $H_M(\cdot,y,\cdot) \le v$. See Section 9.1 in \cite{GrygBook} for details on this property.}, positive fundamental solution to the heat equation $\partial_t u -\Delta_g u = 0 $ on $M$ with $ u(t,\cdot) \to \delta_{\{y\}}$ in the sense of distributions as $t\to 0^+$. The quantity analogous to \eqref{inttheta1} on $\R^n$ was previously studied in \cite{dpfv}, where the authors deal with the study of the fractional $s$-perimeter as $s\to 0^+$. In this case of $M=\R^n$, the limit in \eqref{inttheta1} does not depend on $p$ (whenever it exists); hence, $\theta_E$ is a constant function.

 \subsection{Main results}
 
 One of the main observations of this work is that $\theta_E$ is always an harmonic function on $M$, with values in $[0,1]$, and in general can be non-constant if $M$ does not satisfy the $L^\infty-{\rm Liouville}$ property (see Definition \ref{LinfLiouville}). Moreover, for $E \equiv M$, the function $\theta_M$ encodes the asymptotics of the fractional Laplacian as $s\to 0^+$ on every complete $(M,g)$. 

 \vsp
 The following are the main results of our work.

 \begin{theorem}\label{mainrandom0} Let $(M,g)$ be a complete Riemannian manifold with $\mu(M)=+\infty$, and let $E\subset M$ be a measurable set. Then 
 \begin{itemize}
     
     \item[(i)]  If for some $R>0$ and every $p \in M$, the following limit exists
         \begin{equation}\label{eq: theta limit def}
         \theta_E(p):= \lim_{s\to 0^+} \int_{E \setminus B_R(p)} \mathcal{K}_s(x,p) d\mu(x)  \in [0, 1] \,,
     \end{equation}
     then it is independent of the choice of $R$, and $\theta_E :M\to [0,1]$ is a bounded harmonic function on $M$.  

      \item[(ii)] For $R>0$ and $p \in M$ the limit 
         \begin{equation}\label{volinf M}
         \theta_M(p):=\lim_{s\to 0^+} \int_{M \setminus B_R(p)}  \mathcal{K}_s(x,p) d\mu(x) \in [0,1]
     \end{equation}
    always exists, does not depend on the choice of $R$, and equals
        \begin{equation}\label{argwer}
            \theta_M(p) =\lim_{t\to \infty } \int_M H_M(p,x,t) \, d\mu(x) \,.
        \end{equation}
    Moreover, $\theta_M :M\to [0,1]$ is a bounded harmonic function on $M$.  
 \end{itemize}
 \end{theorem}

 Unless otherwise stated, when we will say \textit{"assume $\theta_E$ exists"} we intend that the limit in \eqref{eq: theta limit def} exists for some $R>0$ and every $p\in M$. We refer to \autoref{sbs: theta existence at diff points} for a brief discussion on the existence/nonexistence of this limit for different points $p$.

\vsp 
Next is the asymptotics of the fractional Laplacian. Note that, on well-behaved ambient spaces, one would expect (as it happens on $\R^n$) that the fractional $(s/2)$-Laplacian tends to the identity as $s\to 0^+$. With the following result, we show that this is not true on general Riemannian manifolds and that the harmonic function $\theta_M$ defined in \eqref{volinf M} encodes how this limit differs from the identity. 
 
\begin{theorem}\label{mainrandom1}
     Let $(M,g)$ be a complete Riemannian manifold with $\mu(M)=+\infty$, and let $\theta_M$ be given by \eqref{volinf M}. Let also $s_\circ \in (0,2)$ and $u \in H^{s_\circ/2}(M)\cap L^\infty(M)$ with bounded support. Then, as $s\to 0^+$ there holds
      \begin{equation}\label{harmasymp}
           (-\Delta)_{\rm Si}^{s/2} u \conv{}  \theta_M u \s a.e. \,\, on \,\, M ,
     \end{equation}
     where $ (-\Delta)_{\rm Si}^{s/2}$ is the singular integral fractional Laplacian \eqref{flap}.
\end{theorem}

With this result, we also make an interesting observation regarding a Riemannian manifold constructed by Pinchover in \cite{PinYehu95}. This Riemannian manifold satisfies the $L^\infty-{\rm Liouville}$ property (see Definition \ref{LinfLiouville}), but it is not stochastically complete, and we show that it satisfies $\theta_M\equiv 0$. We describe the construction of this manifold in Example \ref{exampleshit}. Consequently, there exist complete Riemannian manifolds where the mass of the heat kernel escapes so rapidly that the asymptotics of the fractional Laplacian not only differs from the identity but becomes identically zero, even for regular functions.

\vsp
In the following result, we address the equivalence (actually, equality) of different definitions of the fractional Laplacian on stochastically complete manifolds. Moreover, we also find the asymptotics of the fractional Laplacian on manifolds with finite volume. 

 \begin{theorem}\label{mainrandom2} 
     Let $(M,g)$ be a stochastically complete Riemannian manifold. Let also $s_\circ\in (0,2)$ and $u \in H^{s_\circ/2 }(M)$ (see Definition \ref{fracsobdef}). Then, for all $s<s_\circ $ the three definitions of the fractional Laplacian \eqref{flap}, \eqref{boclap}, and \eqref{speclap} coincide a.e., that is 
     \begin{equation*}
         (-\Delta)^{s/2}_{\rm Si} u = (-\Delta)^{s/2}_{\rm B } u = (-\Delta)^{s/2}_{\rm Spec} u \,.
     \end{equation*}
    Moreover, as $s\to 0^+$
         \begin{align}\label{rand1}
         (-\Delta)^{s/2} u \conv{L^2} u - \frac{1}{\mu(M)} \int_M u \, d\mu  \s\s \textit{if} \,\,\, \mu(M)<+\infty \,,
         \end{align}
         and 
        \begin{align}\label{rand2}
          (-\Delta)^{s/2} u \conv{L^2} u  \s\s \textit{if} \,\,\, \mu(M)=+\infty \,,
         \end{align}
     where $(-\Delta)^{s/2}$ is any of the equivalent fractional Laplacians.
 \end{theorem}

 In proving the previous theorems, we also provide an equivalent characterization of being stochastically complete (see Definition \ref{def stoch comp}) in the case of infinite volume. 
\begin{proposition}\label{limexistence}
Let $(M,g)$ be a complete (possibly weighted) Riemannian manifold with $\mu(M)=+\infty$, and let $\theta_M(p)$ be given by \eqref{volinf M}. If $M$ is stochastically complete, then
\begin{equation}
\label{Stoccompl}
         \theta_M = \lim_{s \to 0^+} \int_{M \setminus B_1(p)} \mathcal{K}_s(x,p) d\mu(x)=1\quad\forall p\in M \,.
     \end{equation}
     Conversely, if there exists $p\in M$ such that 
     \begin{equation}
         \theta_M(p) = \lim_{s \to 0^+} \int_{M \setminus B_1(p)} \mathcal{K}_s(x,p) d\mu(x)=1 \,,
     \end{equation}
     then $M$ is stochastically complete.
\end{proposition}
We will prove this result at the beginning of \autoref{section4}. 
\begin{remark} 
   We believe that Theorem \ref{mainrandom0} could be used to count the dimension of the space of bounded harmonic function on $M$. Something in this direction has already been done by A. Grigor’yan in \cite{Grygbddharm}, where he proves that this dimension equals the maximum number of disjoint massive sets that can be put on $M$. We think that the sets $E$ for which \eqref{eq: theta limit def} is not (identically) zero or one are related to the notion of massiveness and could be used to prove a similar statement. We plan to explore this relationship in future work. 
\end{remark}
As a corollary of the results above we are able to obtain the asymptotics of the fractional perimeter as $s\to 0^+$ in an extremely general setting, generalizing both the existing results \cite{dpfv} for $\R^n$ and \cite{CaCiLaPa21a} for the Gaussian space. Although these outcomes currently stem from broader results obtained in our investigation, we emphasize that the initial motivation behind this research was to explore the asymptotic properties of the fractional perimeter on general Riemannian manifolds.

\vsp 
In particular, with Theorem \ref{mainnoncompact} and \ref{maincompact}, we show that these two known behaviors of the asymptotics, the one of $\R^n$ and the one of the Gaussian space, are essentially the only two possible also in this general setting.

 \begin{theorem}[Infinite volume asymptotics]\label{mainnoncompact}
     Let $(M,g)$ be a complete, stochastically complete Riemannian manifold with $\mu(M)=+\infty$ and such that the $L^\infty-{\rm Liouville}$ property holds (see Definition \ref{LinfLiouville}). Let $\Omega \subset M $ be an open, bounded, connected set with Lipschitz boundary. Let also $E\subset M$ be a measurable set with\footnote{We refer to Definition \ref{fper def loc} for the definition of the fractional perimeter $P_s(\cdot, \Omega)$.} $P_{s_\circ}(E, \Omega) <+\infty$, for some $s_\circ \in (0,1)$, and such that $\theta_E$ exists (see \eqref{eq: theta limit def}). Then
     \begin{itemize}
         \item[(i)] The limit $\lim_{s\to 0^+} \frac{1}{2} P_s(E, \Omega )$ exists and\footnote{Since $\theta_{M\setminus E}=1-\theta_E$ in this case.}
      \begin{align*}
        \lim_{s\to 0^+} \frac{1}{2} P_s(E, \Omega ) &= (1-\theta_E) \mu(E\cap \Omega) + \theta_E \mu(E^c \cap \Omega) \\ &=\theta_{M\setminus E}   \mu(E\cap \Omega) + \theta_E  \mu(E^c \cap \Omega) \,.
    \end{align*}
    \item[(ii)] Conversely, if $\mu(\Omega\cap E) \neq \mu(\Omega \setminus E)$ and the limit $\lim_{s\to 0^+} \frac{1}{2} P_s(E, \Omega )$ exists, then the limit in \eqref{eq: theta limit def} exists and there holds
    \begin{equation*}
        \theta_E=\frac{\lim_{s\to 0^+} \frac{1}{2} P_s(E, \Omega )-\mu(E\cap \Omega)}{\mu(\Omega \setminus E)-\mu(E\cap \Omega)} \,.
    \end{equation*}
    
    \item[(iii)] If $\mu(\Omega\cap E) = \mu(\Omega \setminus E)$ then the limit $\lim_{s\to 0^+} \frac{1}{2} P_s(E, \Omega )$ always exists and 
    \begin{equation*}
        \lim_{s\to 0^+} \frac{1}{2} P_s(E, \Omega ) = \mu(\Omega\cap E) = \mu(\Omega \setminus E) \,.
        \end{equation*}
     \end{itemize}
 \end{theorem} 

\begin{remark} Without the assumption of stochastic completeness of $M$ the situation can be different. We will describe in Example \ref{exampleshit} a complete Riemannian manifold $N$, with the $L^\infty-{\rm Liouville}$ property but not stochastically complete such that $\lim_{s\to 0^+}  P_s(E )=0$ for every subset $E\subset N$. 
    
\end{remark}

 \begin{theorem}[Finite volume asymptotics]\label{maincompact}
Let $(M,g)$ be a complete Riemannian manifold with $\mu(M)<+\infty$, and let $\Omega \subset M$ be an open and connected set with Lipschitz boundary. If for some set $E \subset M$ there exists $s_\circ \in (0,1)$ such that $P_{s_\circ}(E, \Omega) <+\infty $, then the limit $\lim_{s\to 0^+} \frac{1}{2} P_s(E, \Omega )$ exists and 
\begin{equation*}
    \lim_{s\to 0^+} \frac{1}{2} P_s(E, \Omega ) = \frac{1}{\mu(M)}\Big( \mu(E )\mu(E^c \cap \Omega) + \mu(E\cap \Omega)\mu(E^c \cap \Omega^c)  \Big).
\end{equation*}
 \end{theorem}

\subsection{The fractional perimeter on Riemannian manifolds}
It was recently pointed out in \cite{CFS23} a canonical definition of the fractional $s$-perimeter on every closed Riemannian manifold $(M,g)$: this boils down to giving a canonical definition of the fractional Sobolev seminorm $H^{s/2}(M)$ for $s\in (0,1)$. Consider a closed (even though we will deal with general complete ones), connected Riemannian manifold $(M,g)$ with $n\ge2$. In \cite{CFS23} the authors show that a canonical definition of the fractional Sobolev seminorm $H^{s/2}(M)$ can be given in at least four equivalent (up to absolute constants) ways:

 \begin{itemize}
     \item[\it (i)] By the {\em singular integral}  
     \begin{equation}\label{aaa}
 [u]^2_{H^{s/2}(M)} := \iint_{M\times M}(u(x)-u(y))^2 \mathcal{K}_s(x,y) \,d\mu(x) \,d \mu(y) \,,
 \end{equation}
where $\mathcal{K}_s(x,y)$ is given by \eqref{singkerdef}.
     \item[\it (ii)] Following the {\em Bochner definition} of the fractional Laplacian
     \begin{equation}\label{boclap}
   (-\Delta)^{s/2}_{\rm B} \, u = \frac{1}{\Gamma(-s/2)} \int_{0}^{\infty} (e^{t\Delta}u-u)\frac{dt}{t^{1+s/2}}\, ,
\end{equation}
via 
\begin{equation*}
    [u]^2_{H^{s/2}(M)}  = 2 \int_{M} u (-\Delta)^{s/2}_{\rm B} u \, d\mu  \,.
\end{equation*}
     \item[\it (iii)] By  {\em spectral theory}, one can set
\begin{equation}\label{ghfghfg}
    [u]^2_{H^{s/2}(M)} = \sum_{k\ge 1} \lambda_k^{s/2} \langle u,\phi_k \rangle^2_{L^2(M)}
\end{equation}
 where $\{\phi_k\}_k$ is an orthonormal basis of eigenfunctions of the Laplace-Beltrami operator $(-\Delta_g)$ and $\{\lambda_k\}_k$ are the corresponding eigenvalues. Note that for $s=2$ this gives the usual $[u]^2_{H^{1}(M)}$ seminorm.
     \item[\it (iv)] Considering a {\em Caffarelli-Silvestre type extension (cf. \cite{CafSi,BGS, ChangGon})}, namely, a degenerate-harmonic extension problem in one extra dimension. One can set
     \begin{equation*}
         [u]^2_{H^{s/2}(M)} = \inf \left\{ \int_{M\times [0,\infty)} z^{1-s} |\widetilde {\nabla}   U(x,z)|^2 \, d\mu(x) dz \s {\rm s.t.} \s U(x, 0)=u(x) \right\}.
     \end{equation*}
Here $\widetilde {\nabla}$ denotes the Riemannian gradient of the manifold $\widetilde M = M\times [0,\infty)$, with respect to natural product metric, and the infimum is taken over all the extensions $U \in X$, where $ X = H^1(\widetilde{M} \,; z^{1-s} d\mu dz)$ is the classical weighted Sobolev space of the functions $ U \in L^2(\widetilde{M} \,; d\mu)$ with respect to the measure $d\mu = z^{1-s}d\mu dz$ that admit a weak gradient $  \widetilde {\nabla} U \in L^2(\widetilde{M} \,;  d\mu)$. 
 \end{itemize}

The spectral definition $(iii)$ can be extended to manifolds that are not closed, where the spectrum of the Laplacian is not discrete. Nevertheless, the equivalence between $(i)$ and $(iv)$ also holds on many (but not every) complete Riemannian manifolds, which are not necessarily compact. For example, a lower Ricci curvature bound is sufficient. See \cite{BGS} for general conditions for which the equivalence of $(i) \iff (iv)$ holds. Moreover, under suitable assumptions on $u$, the equivalence between $(i)$ and $(ii)$ holds if and only if $M$ is stochastically complete; we will treat this equivalence in \autoref{flapsection}.

\vsp
 Since in the present work, we aim to study the asymptotics of the fractional $s$-perimeter on complete Riemannian manifolds (not necessarily closed or with curvature bounded below), we work with the singular integral definition \eqref{aaa} since it extends naturally to the case of general manifolds and weighted manifolds. Then, the fractional $s$-perimeter on a Riemannian manifold is naturally defined by means of the fractional Sobolev seminorm. 
 
 \vsp
 Here and in the rest of the work, $(M,g)$ will denote a general complete, connected Riemannian manifold, and hence also geodesically complete. We denote by $d\mu$ its Riemannian volume form and by $H_M(x,y,t)$ the heat kernel of $(M,g)$. To see how to build the heat kernel on a general (weighted) manifold, see the classical reference \cite{GrygBook}. Moreover, we denote by $B_R(p) \subset M$ the geodesics ball on $M$ and by $\B_R(0) \subset \R^n$ the one on $\R^n$.
 
 \begin{definition}\label{fracsobdef}
    Let $(M,g)$ be a complete Riemannian manifold and $s \in (0, 2)$. Then, we set 
     \begin{equation*}
H^{s/2}(M) :=\big\{u\in L^2(M)  \text{ : } [u]^2_{H^{s/2}(M)}<\infty \big \} \,,
\end{equation*}
where
\begin{equation*}\label{sfgs}
    [u]^2_{H^{s/2}(M)} := \iint_{M\times M} (u(x)-u(y))^2 \mathcal{K}_{s}(x,y) \, d\mu(x) d\mu(y) \,,
\end{equation*}
and $\mathcal{K}_{s}$ is defined as in \eqref{singkerdef}. 
 \end{definition}
Moreover, we will use the singular integral
\begin{equation}\label{flap}
    (-\Delta)^{s/2}_{\rm Si} u (x) := {\rm P.V.} \,\frac{1}{|\Gamma(-s/2)|} \int_{M} (u(x)-u(y)) \mathcal{K}_s(x,y) \, d\mu(y)
\end{equation}
as our main definition of {\em "the fractional Laplacian"} on $M$. We stress that in the general setting of complete Riemannian manifolds, this integro-differential operator cannot be regarded as a fractional power of the Laplacian in any reasonable sense. In particular:
\begin{itemize}
 \setlength\itemsep{3pt}
    \item If $M$ is not stochastically complete (see Definition \ref{def stoch comp}), then $(i)$ and $(ii)$ do not coincide. In this case, since $e^{t\Delta} 1  \neq 1$, the Bochner fractional Laplacian $(ii)$ of a constant does not equal zero. In particular, defining the fractional Sobolev seminorm with $(ii)$ would imply that the $s$-perimeter is not invariant under complementation $P_s(E) \neq P_s(E^c) $. Nevertheless, with our definition via the singular integral $(i)$, one has that the seminorm of a constant is always zero, and hence, in this work, the fractional perimeter is always invariant under complementation.  
    \item The semigroup property $ (-\Delta)^{\alpha + \beta} = (-\Delta)^\alpha \circ (-\Delta)^\beta  $ also fails in general for our definition \eqref{flap}. Indeed, one can see that the equivalence $(i) \iff (iv)$ above is sufficient for the semigroup property to hold. For example, a Ricci curvature lower bound would be sufficient. See \cite{BGS} for many sufficient conditions for the equivalence $(i) \iff (iv)$. 
\end{itemize} 

\begin{definition}\label{fper def glob}
     For a measurable set $E \subset M$, we define the fractional $s$-perimeter of $E$ on $(M,g)$ as
\begin{equation*}
    P_s(E) := [\chi_E]^2_{H^{s/2}(M)} = 2\iint_{E\times E^c} \mathcal{K}_s(x,y) \, d\mu(x)d\mu(y) \,,
\end{equation*}
where $[\,\cdot \, ]^2_{H^{s/2}(M)}$ is defined by \eqref{aaa} and $\chi_E$ is the characteristic function of $E$. 
 \end{definition}

 Apart from the above definition of the fractional perimeter of a set $E$ on the entire $M$, we will also consider its localized version. For $A,B \subset M$ disjoint and measurable sets, let 
\begin{equation*}
    \mathcal{J}_s(A,B) := \iint_{A \times B} \mathcal{K}_s(x,y) \, d\mu(x) d\mu(y)
\end{equation*}
be the $s$-interaction functional between the sets $A$ and $B$.

\begin{definition}\label{fper def loc}
Let $(M,g)$ be a complete Riemannian manifold, and let $\Omega \subset M$ be an open and connected set with Lipschitz boundary. We define the $s$-perimeter of $E$ in $\Omega $ as

\begin{align*}
   \frac{1}{2} P_s(E, \Omega) & := \iint_{(M \times M ) \setminus ( \Omega^c \times \Omega^c )} (\chi_E(x)-\chi_E(y))^2 \mathcal{K}_s(x,y) \, d\mu(x) d\mu(y) \\[1ex] & \vspace{5pt} =  \J_s(E\cap \Omega, E^c \cap \Omega ) +  \J_s(E\cap \Omega, E^c \cap \Omega^c ) +  \J_s(E\cap \Omega^c, E^c \cap \Omega ) \,.
\end{align*}

\end{definition}

For any measurable $E\subset M$, it is clear by the definition above that $P_s(E, \Omega)=P_s(E^c, \Omega)$, that $P_s(E, M) = P_s(E) = [\chi_E]_{H^{s/2}(M)}^2$ and also that $P_s(E, \Omega) = P_s(E) $ if $E\subset \Omega$ or $E^c \subset \Omega$.

\vsp

 \begin{remark} The hypothesis $P_{s_\circ}(E, \Omega) <+\infty $ for some $s_\circ \in (0,1)$ cannot be removed in neither of these results. Indeed, in \cite[Example 2.10]{dpfv} the authors exhibit a bounded set $E \subset \R$ such that $P_s(E)= +\infty $ for all $s\in (0,1)$.
 \end{remark} 
 
\begin{remark}\label{constantinRn}
Note that, taking $M=\R^n$ with its standard metric in Theorem \ref{mainnoncompact} gives $\mathcal{K}_s(x,p)=\frac{\beta_{n,s}}{|x-p|^{n+s}}$, where
\begin{equation*}
    \beta_{n,s} = \frac{s 2^{s-1} \Gamma\Big(\tfrac{n+s}{2}\Big)}{\pi^{n/2}\Gamma(1-s/2)} \,.
\end{equation*}
Hence
\begin{equation*}
         \theta_{\R^n} = \lim_{s\to 0^+}  \int_{\R^n \cap B_R^c(p)} \frac{\beta_{n,s}}{|x-p|^{n+s}} dx  = \frac{\Gamma(\frac{n}{2})}{2\pi^{n/2}} \lim_{s\to 0^+} s \int_{\R^n \cap B_1^c(0)} \frac{1}{|x|^{n+s}} dx = \frac{\Gamma(\frac{n}{2})}{2\pi^{n/2}} \alpha_{n-1} =1\,,
     \end{equation*}
where $\alpha_{n-1}$ is the volume of the unit sphere $\Sp^{n-1}$. Moreover, analogously for $E\subset \R^n$ (if the limit exists)
\begin{equation*}
         \theta_E = \lim_{s\to 0^+} \int_{E \cap B_1^c(0)} \frac{\beta_{n,s}}{|x|^{n+s}} dx = \frac{1}{\alpha_{n-1}} \lim_{s\to 0^+} s \int_{E \cap B_1^c(0)} \frac{1}{|x|^{n+s}} dx  \in [0, 1] \,,
     \end{equation*}
     which is (up to the absolute multiplicative constant $\alpha_{n-1}^{-1}$) what is denoted by $\alpha(E)$ in \cite{dpfv}. Hence, we see that in the case of the Euclidean space our result Theorem \ref{mainnoncompact} recovers the one in \cite{dpfv}.
\end{remark}

\begin{remark}
    Note that, as $s\to 0^+$, the constant in \eqref{singkerdef} satisfies 
\begin{equation*}
    \frac{1}{|\Gamma(-s/2)|} = \frac{s/2}{\Gamma(1-s/2)} \sim \frac{s}{2} \,.
\end{equation*}
 We will use this fact many times in the computations of the asymptotics. 
\end{remark}

The paper is divided as follows. In \autoref{section2} we recall some facts and definitions that we will need regarding the heat kernel and harmonic functions on general complete manifolds. In \autoref{section3} we prove the all the main results stated at the beginning of the introduction. Then, building on our main results, in \autoref{section4} and \autoref{section5} we prove Theorem \ref{mainnoncompact} and Theorem \ref{maincompact} regarding the asymptotics of the fractional perimeter in infinite volume and finite volume respectively.

\vsp
   Lastly, in \autoref{extension section} we explain why our results hold in a much more general setting than the one of Riemannian manifolds, namely {\rm RCD} spaces. We could have proved our theorem directly in this generality, but we believe that a presentation for Riemannian manifolds is easier to follow and already captures all the possible (two) behaviors of the limit of the asymptotics: this also allows us to present different proofs. For these reasons, we have moved everything regarding non-smooth spaces to \autoref{extension section}.

   \vspace{10pt}

\textbf{Acknowledgements.} 
We are very thankful to Diego Pallara and Nicola Gigli for useful discussions and comments on a draft of this work and to Alexander Grigoryan for having shown us how to prove part $(ii)$ of Proposition \ref{gryg}. We also thank the Fields Institute in Toronto, which is where the two authors first met, for the kind hospitality during the thematic program "Nonsmooth Riemannian and Lorentzian Geometry" in the fall semester 2022 . 

	\section{The heat kernel on Riemannian manifolds}\label{section2}

 Let us start by recalling a few classical definitions and results.

 \begin{definition}[Stochastical completeness]\label{def stoch comp} We call a Riemannian manifold $(M,g)$ stochastically complete if, for every $t>0$ and for every $p\in M$
 \begin{equation}\label{asrgas}
     \int_M H_M(x,p,t) \, d\mu(x) = 1 \,.
 \end{equation} 
 \end{definition}
For equivalent definitions of stochastical completeness, one can refer to the manuscript \cite{GrygBook} or to the more recent \cite{Grillo20} and \cite{grillo2023general}.
 \begin{lemma}
     \label{Mass decay}
     Let $(M,g)$ be a complete Riemannian manifold, then for every $p\in M$
     \begin{equation*}
       \mathcal{M}(t,p) =  \int_M H_M(x,p,t)\, d\mu(x) \s \textit{is nonincreasing in $t$.}
     \end{equation*}
 \end{lemma}
 \begin{proof}
     The proof is an easy consequence of the semigroup property. Indeed, for $t>s$ we can write 
     \begin{equation*}
        H_M(z,p,t) = \int_{M} H_M(z,x,t-s) H_M(x,p,s) \,d\mu(x) .
     \end{equation*}
     Integrating in $d\mu(z)$, using Fubini's theorem and the fact that $\int_{M}H_M(z,x,t-s)d\mu(x)\leq 1$ we get
     \begin{equation*}
         \int_M H_M(z,p,t)d\mu(z)\leq\int_M H_M(x,p,s)d\mu(x),
         \end{equation*}
         which is the thesis.
 \end{proof}
 Note that, because of Lemma \ref{Mass decay}, being stochastically complete is equivalent to the fact that \eqref{asrgas} holds for one single time $t=t_\circ>0$.

 \begin{theorem}[Yau]\label{dfsdfs} Let $(M,g)$ be a complete Riemannian manifold. Then every $L^2(M)$ harmonic function is constant.
\end{theorem}
\begin{proof}
    Let $u\in L^2(M)$ be harmonic. It is a standard result by Yau (see for example \cite[Lemma 7.1]{Liga}) that, on every complete Riemannian manifold $M$, the Caccioppoli-type inequality
    \begin{equation}\label{ndfhgnf}
        \int_{B_R(p)} |\nabla u|^2 \, d\mu \le \frac{4}{R^2} \int_{B_{2R}(p)} |u|^2 \, d\mu 
    \end{equation}
    holds. Since $u\in L^2(M)$, letting $R\to \infty$ gives that $u$ is constant.
\end{proof}

 \begin{definition}[$L^\infty-{\rm Liouville}$ property]\label{LinfLiouville} We say that a Riemannian manifold $(M,g)$ has the $L^\infty-{\rm Liouville}$ property if every bounded harmonic function on $M$ is constant.
 \end{definition}

 Since the validity of the $L^\infty-{\rm Liouville}$ property will be a key feature in our result for infinite volume, we shall recall few conditions that imply this property. See \cite{Gryg22222} for more general conditions under which the $L^\infty-{\rm Liouville}$ property holds.

 \begin{proposition}
     Let $(M,g)$ be a complete Riemannian manifold. Then, each of the following properties implies the $L^\infty-{\rm Liouville}$ property for $M$:
     \begin{itemize}
     \setlength\itemsep{3pt}
         \item[(i)] $\Ric_M\ge 0$.
         \item[(ii)] $\mu(B_R(p))/R^2\to 0 $ as $R\to \infty $ for some (and hence any) $p\in M$.
         \item[(iii)]  There exists a metric $\widetilde g$ on $M$ and $K\subset M$ compact such that $\widetilde g = g $ in $M\setminus K$ and $(M,\widetilde g)$ has the $L^\infty-{\rm Liouville}$ property. 
     \end{itemize}
 \end{proposition}
 \begin{proof} To show $(i)$ we just need to apply the $L^\infty-{\rm Lip}$ regularization of \eqref{Lip regularization}, that we state in general for ${\rm RCD}$ spaces in \autoref{extension section} and we give a simple proof at the end of the Appendix. Indeed let $u\in L^\infty(M)$ be such a function: we can clearly assume $\|u\|_{L^\infty}=1$ so that we have $\|\nabla e^{t\Delta}u\|_{L^\infty}\leq C/\sqrt{t}$. The previous estimate tells us that $\|\nabla e^{t\Delta}u\|_{L^\infty}\to 0$ as $t\to\infty$ so that $e^{t\Delta}u\to {\rm const}$ weakly star in $L^\infty(M)$. However, we also know that $e^{t\Delta}u=u$ for every $t\in (0,\infty)$ because of the uniqueness of the solution of the heat equation (due to stochastical completeness which holds in the presence of a lower Ricci curvature bound) and this means that $u$ has to be constant.
 
 Part $(ii)$ follows from Yau's estimate \eqref{ndfhgnf} letting $R\to \infty$. Lastly, the proof of part $(iii)$ is contained in \cite[Proposition 4.2]{Gryg22222} and \cite[Theorem 5.1]{Gryg22222}.
 \end{proof}
 
 Notice that $\Ric_M\ge -K$ for some $K>0$ is not sufficient for the $L^\infty-{\rm Liouville}$ property to hold since there exist non-constant bounded harmonic functions on the hyperbolic space $\mathbb{H}^n$. Since $\mathbb{H}^n$ is stochastically complete, this means that stochastical completeness does not imply the $L^\infty-{\rm Liouville}$ property. Moreover, quite surprisingly, stochastical completeness of $M$ is not implied by the $L^\infty-{\rm Liouville}$ property. The first example of such a manifold was constructed by Pinchover in \cite{PinYehu95}, we briefly explain this construction in Example \ref{exampleshit}.  
 
\vsp
 In the next lemma, we give the proof of a result that, perhaps, is well-known to the experts, but we could not find an appropriate reference. The case $\mu(M)<+\infty$ is stated in \cite{GrygBook} as Exercise 11.21. 

 \vsp 
 We stress that these results easily extend to the context of weighted Riemannian manifolds.

\begin{lemma}\label{gryg}
    Let $(M,g)$ be a complete, connected Riemannian manifold. Then
    \begin{itemize}
        \item[(i)] If $\mu(M)<+\infty$, then for all $x,y\in M$
        \begin{equation*}
           \lim_{t \to +\infty} H_M(x,y,t) =\frac{1}{\mu(M)} \,,
        \end{equation*}
        and the convergence is uniform in every bounded $\Omega \subset M$, that is
         \begin{equation*}
           \lim_{t \to +\infty} \sup_{ x,y \in \Omega} \left| H_M(x,y,t)-\frac{1}{\mu(M)} \right|=0 \,.
        \end{equation*}
        \item[(ii)] If $\mu(M)=+\infty$, then for all $x,y\in M$
        \begin{equation*}
           \lim_{t \to +\infty} H_M(x,y,t) =0 \,,
        \end{equation*}
        and the convergence is uniform in every bounded $\Omega \subset M$, that is
         \begin{equation*}
           \lim_{t \to +\infty} \sup_{ x,y \in \Omega} H_M(x,y,t) =0 \,.
        \end{equation*}
        Moreover, for every fixed $p\in M$ there holds also
            \begin{equation}\label{ryw235w4}
           \lim_{t \to +\infty} \sup_{ x \in M} H_M(x,p,t) =0 \,.
        \end{equation}

    \end{itemize}
\end{lemma}

\begin{proof}
    To prove the result we use standard spectral theory. Let us first do the case $\mu(M)=+\infty$. The spectrum of the Laplacian $\sigma(-\Delta)$ is contained in $[0,\infty)$ and by Theorem \ref{dfsdfs} we know that the eigenspace of $\lambda=0$ contains no constant function except for the function identically $0$. 
    
    Let $ \{E_\lambda \}_{\lambda \ge 0} $ be the spectral resolution of the Laplacian, then for every $f\in L^2(M)$ (here $\lp \,\cdot \, ,\,\cdot \,\rp$ denotes the $L^2(M)$ scalar product)
    \begin{equation*}
        \lp e^{t\Delta}f,f \rp=\int_0^\infty e^{-t\lambda}d\lp E_\lambda f,f\rp.
    \end{equation*}
    
 Since $\lim_{t\to\infty}e^{-\lambda t}=\chi_{\{0\}}(\lambda)$ we can apply dominated convergence to deduce that 
\begin{equation*}
    \lim_{t\to \infty} \lp e^{t\Delta}f,f \rp =\lp E_0f,f \rp \,,
\end{equation*}
and since $E_0$ projects onto the eigenspace of $\lambda=0$, made only by the constant function identically zero, we get
 \begin{equation}\label{eq 1 gryg lem}
        \lim_{t\to\infty} \lp e^{t\Delta}f,f \rp = 0 \,.
    \end{equation}
    
Now note that for all $f,g \in L^2(M)$ we have $| \lp e^{t\Delta}f, g \rp |=|\lp e^{t/2\Delta}f,e^{t/2\Delta} g  \rp |$. Thus by Cauchy-Schwartz
    \begin{align*}
        \lp e^{t\Delta}f , g \rp =    \lp e^{t/2 \Delta}f , e^{t/2 \Delta}g \rp & \le \| e^{t/2 \Delta}f \|_{L^2} \| e^{t/2 \Delta} g \|_{L^2} \\ & = \lp e^{t\Delta}f ,f \rp \lp e^{t\Delta} g, g \rp \,.
    \end{align*}
    Taking the supremum over $g \in L^2(M)$ with $\| g \|_{L^2} \le 1$ and sending $t\to \infty$ gives that $e^{t\Delta}f \to 0 $ strongly in $L^2(M)$. Since this holds for all $f\in L^2(M)$, this implies $H_M(\cdot,y,t) \to 0 $ in $L^2(M)$ as $t\to \infty$. 

\vsp
 Now, by a local parabolic Harnack inequality, we are able to turn this convergence into pointwise convergence that is actually locally uniform. Indeed for $p \in M$, $R \ll 1$ to be chosen depending on $p$, and $t\ge 10$, taking $f=\chi_{B_R(p)}$ above gives
     \begin{equation*}
    \lp e^{t\Delta}\chi_{B_R(p)},\chi_{B_R(p)} \rp = \int_{B_R(p)} \int_{B_R(p)} H_M(x,y,t) \, d\mu(x) d\mu(y) \ge \mu(B_R(p))^2 \inf_{x,y \in B_R(p)}  H_M(x,y,t) \,.
     \end{equation*}
     
     By the parabolic Harnack inequality (see Remark \ref{harnackprop} after this proof) applied two times
     \begin{align*}
         \inf_{x,y \in B_R(p)}  H_M(x,y,t) & \ge C^{-1} \inf_{x \in B_R(p)} \sup_{y \in B_R(p)} H_M(x,y,t-1/2) \\ & \ge C^{-1} \sup_{x \in B_R(p)} \inf_{y \in B_R(p)} H_M(x,y,t-1/2) \\ &  \ge C^{-2} \sup_{x,y \in B_R(p)} H_M(x,y,t-1) \,, 
     \end{align*}
     for some $C>0$ depending on $B_R(p) \subset M $ but independent of $t$. Hence 
     \begin{equation*}
        \sup_{x,y \in B_R(p)} H_M(x,y,t) \le C(B_R(p))  \lp e^{(t+1)\Delta}\chi_{B_R(p)},\chi_{B_R(p)}\rp \to 0 \,,
     \end{equation*}
     as $t\to \infty$. Covering any bounded set with small balls allows us to infer the desired local uniform convergence.

     \vsp
     We are left to prove \eqref{ryw235w4}. By the properties of the heat kernel, we have
    \begin{equation*}
        H_M(p,p,t)=\int_M H_M^2(p,z,t/2)d\mu(z)=\|H_M(p,\cdot,t/2)\|_{L^2(M)}^2.
    \end{equation*}
    Moreover
    \begin{equation*}
        H_M(x,p,t)=\int_M H_M(x,z,t/2)H_M(p,z,t/2)d\mu(z)\leq \sqrt{H_M(p,p,t)}\|H_M(x,\cdot,t/2)\|_{L^2(M)},
    \end{equation*}
    which concludes the proof if we are able to show that $ \sup_{x\in M} \|H_M(x,\cdot,t/2)\|_{L^2(M)}$ is bounded as $t\to\infty$. 
    
    However since $H_M(x,y,t)=e^{(t-1)\Delta}(H_M(x,\cdot,1))(y)$ and we have the contraction estimate $\|e^{s\Delta}(f)\|_{L^2(M)}\leq\|f\|_{L^2(M)}$ for every $s\in(0,\infty)$ and for every $f\in L^2(M)$ we can write
    \begin{equation*}
        \|H_M(x,\cdot,t)\|_{L^2(M)}=\|e^{(t-1)\Delta}(H_M(x,\cdot,1))\|_{L^2}\leq\|H_M(x,\cdot,1)\|_{L^2}\quad\forall t>1.
    \end{equation*}
    Therefore, we reach the sought conclusion. This concludes the proof of $(ii)$. 

\vsp
Now assume $\mu(M)<+\infty$. Since the proof in this case is almost identical to the one for infinite volume, we just sketch the argument, highlighting the differences. The only essential difference is that in the case $\mu(M)<+\infty$, the eigenspace relative to $\lambda=0$ is made only by the constant function $\mu(M)^{-1/2}$. Hence
\begin{equation*}
    E_0f = \frac{1}{\mu(M)}\int_M f \, d\mu =: \fint_M f  \,,
\end{equation*}
and in place of \eqref{eq 1 gryg lem} we get 
\begin{equation*}
     \lim_{t\to \infty} \Big \lp e^{t\Delta}f - \fint_M f ,f  \Big \rp \to 0 \,.
\end{equation*}

From here, the proof proceeds exactly the same showing that $H_M(\cdot ,y,t)-1/\mu(M) \to 0 $ strongly in $L^2(M)$. Then, one can turn the convergence into pointwise and locally uniform by a similar argument with the parabolic Harnack inequality. 

\vsp
Indeed, the function $v:=(H_M(\cdot ,y,t)-1/\mu(M))_+$ (where $(f)_+$ denotes the positive part of $f$) is a nonnegative subsolution to the heat equation. Then, by the parabolic version of the Moser-Harnack inequality (see, for example, \cite[Theorem 5.1]{Saloff-Coste}) we have (here $C$ depends on $R$ and the geometry of $M$ in $B_{2R}(p)$)
\begin{equation}\label{dgfhjdh}
    \sup_{[t+R^2/2, t+R^2]}  v^2 \le C \int_t^{t+R^2} \int_{B_R(p)} |v|^2 \, d\mu \to 0 \,.
\end{equation}
Hence $\limsup_{t \to \infty} H_M(\cdot ,y,t ) \le 1/\mu(M)$. Argiung similarly with the negative part gives also the $\liminf$ inequality, and hence the pointwise convergence. The fact that the convergence is uniform follows from \eqref{dgfhjdh}.

\end{proof}

\begin{remark}
   Since $\| H_M(x,\cdot,t) \|_{L^1(M)} \le 1 $ and $\| H_M(x,\cdot,t) \|_{L^\infty(M)} \to 0 $ as $t\to \infty$, we conclude that also $\| H_M(x,\cdot,t) \|_{L^p(M)} \to 0 $ for any $p\in (1, \infty]$. The convergence to zero in $L^1(M)$ is clearly prevented if $M$ is stochastically complete.
    \end{remark}
    
\begin{remark}\label{harnackprop} We emphasize that we have used only a local (non-uniform) Harnack inequality in $B_R(p) \subset M $, that is where the constant is allowed to depend on the point $p$ and radius $R$. This is clear since, for fixed $p\in M$ one can take $R \ll 1$ such that, in normal coordinates at $p$, the metric coefficients satisfy $ \|g_{ij}-\delta_{ij} \|_{C^2(B_R(p))} \le 1/100$. Then, any solution $u:B_R(p) \to \R $ to the heat equation on $M$ satisfies (in coordinates) 
\begin{equation*}
    u_t-Lu =0 \,, \quad \textit{in } \B_R(0) \times (0, + \infty) \,,
\end{equation*}
where $-L$ is a uniformly elliptic operator with uniformly bounded coefficients. Hence, by the standard Harnack inequality on $\R^n$ one can conclude the local estimate. 

\vsp
On the other hand, for general Riemannian manifolds, a uniform Harnack inequality (that is, with the constant independent of $R$ and the point $p$) fails, and strong assumptions are required for it to hold. Actually, the validity of a volume doubling property and a uniform Poincarè inequality is equivalent to the uniform Harnack inequality, this was first proved in \cite{SaloffHarnack}.
\end{remark}

\begin{remark}\label{heatsolconv}
   One can turn the previous local uniform convergence in \eqref{ryw235w4} into the convergence of solutions of the heat equation. Indeed, in the case $\mu(M)=+\infty$, since $H_M(\cdot,p,t)$ converges uniformly to zero we get (by dominated convergence)
   \begin{equation*}
       e^{t\Delta}f(y)=\int_M H_M(x,y,t)f(x)d\mu(x)\to 0\quad{\rm as}\;{t\to\infty} \,,
   \end{equation*}
   for every $y\in M$ and $f\in L^1(M)$.
\end{remark}

\section{Proof of the main results}\label{section3}

First, we shall briefly comment on the following quantity
\begin{equation*}
    \alpha(E) = \lim_{s \to 0^+} s \int_{E \setminus B_1(0)} \frac{1}{|y|^{n+s} } \, dy ,
\end{equation*}
introduced by Dipierro, Figalli, Palatucci, and Valdinoci in \cite{dpfv} as a measure of the behavior of the set $E$ near infinity, and which is (up to a dimensional constant) the limit in \eqref{volinf M} in the case $M=\R^n$ with its standard metric. This quantity is invariant by rescaling of $E$ and, at first, can be thought as a measure of "how conical" is $E$ near infinity. Indeed, if the blow-down $ E/\lambda$ converges in $L^1_{\rm loc} (\R^n)$ to a regular cone $E_\infty$ as $\lambda \to \infty$, then $\alpha(E) = \mathcal{H}^{n-1} (E_\infty \cap \Sp^{n-1})$. Nevertheless, the fact that this limit exists in not equivalent to having a conical blow-down. Indeed, one can easily construct examples where the limit in $\alpha(E)$ exists but the blow-downs of $E$ converge to two different cones along two different subsequences. 

\vsp
Finally, the authors in \cite{dpfv} refer to $\alpha(E)$ as the weighted volume towards infinity of the set $E$; however in light of our results and description, it would be more appropriate to call this quantity $\emph{heat density over E}$. Indeed, $\alpha(E)$ represents the fraction of heat kernel that flows through the set towards infinity (this explains why $\theta_M \equiv 1$ on stochastically complete manifolds). 

\vsp
Because of this intuitive reason, the limit in the definition of $\alpha(E)$ needs not to exist in general if $E$, for example, oscillates between two cones near infinity. See \cite[Example 2.8]{dpfv} for the construction of such an example.  

\vsp
On a Riemannian manifold, a similar quantity is needed but, since no canonical origin (as in $\R^n$) is present, the singular kernel $1/|y|^{n+s}$ has to be replaced with $\mathcal{K}_s(y,p)$ and it has to be proved if and when the limit \eqref{eq: theta limit def} becomes independent of $p \in M$. On Riemannian manifolds, this property of the limit being independent of the base point $p$ turns out to be quite delicate and, as a consequence of Theorem \ref{mainrandom0}, we will see that is implied by the $L^\infty-{\rm Liouville}$ property of Definition \ref{LinfLiouville}. 

\begin{definition}[Heat density of a set]
     Let $E\subset M$ be a measurable set with $P_{s_\circ}(E, \Omega) <+\infty$ for some $s_\circ \in (0,1)$. We define, for every $p \in M$ and $R>0$, the \emph{heat density of $E$} as the following limit
     \begin{equation*}
         \theta_{E}(p,R):=\lim_{s\to 0}  \int_{E \setminus B_R(p)} \mathcal{K}_s(x,p) d\mu(x) \,,
     \end{equation*}
     when it exists. At this level, this may depend on $p$ and $R$.
 \end{definition}
 
Note that, at this point, it is not even clear whether the limit \eqref{volinf M} of the heat density $\theta_M$ of the whole $M$ exists or is different from zero. For example, as a consequence of the proof of Theorem \ref{mainnoncompact}, if there were complete Riemannian manifolds with $\mu(M)=+\infty$ and $\theta_M \neq 1$, then we would see the asymptotic 
\begin{equation*}
        \lim_{s\to 0^+} \frac{1}{2} P_s(E, \Omega ) = (\theta_M-\theta_E) \mu(E\cap \Omega) + \theta_E \mu(E^c \cap \Omega) 
    \end{equation*}
holding (even when $\theta_M \neq 1$ ), and if $\theta_M=0$ this would mean that there are Riemannian manifolds where the asymptotic of the fractional $s$-perimeter of any set $E$ is zero. These type of Riemannian manifolds exist; since $\theta_M \neq 1$ in this case, they are not stochastically complete. We will describe such a manifold in Example \ref{exampleshit}.

\vsp
Now, we show that this does not happen if $M$ is stochastically complete: the limit \eqref{volinf M} always exists, and it is equal to one. Actually, more is true: if there is a point $p\in M$ for which the limit is $1$, then the manifold is stochastically complete. Indeed, this is the statement of Proposition \ref{limexistence} that we now prove.

\begin{proof}[Proof of Proposition \ref{limexistence}]
Note that since $\mu(M)=+\infty$ we have $\mu(M\setminus B_1(p))>0$. We want to compute the following
   \begin{equation*}
       \lim_{s \to 0^+} \frac{s}{2}\int_{M \setminus B_1(p)} \int_0^\infty H_M(x,p,t)\frac{dt}{t^{1+s/2}} d\mu(x).
   \end{equation*}

   \vsp\noindent
   \textbf{Claim 1.} There holds
    \begin{equation*}
        \lim_{s\to 0^+}\frac{s}{2}\int_{M\setminus B_1(p)}\int_0^1H_M(x,p,t)\frac{dt}{t^{1+s/2}}d\mu(x)=0.
    \end{equation*}
    Indeed, this directly follows by writing 
    \begin{equation*}
       \lim_{s\to 0^+}\frac{s}{2}\int_{M\setminus B_1(p)}\int_0^1H_M(x,p,t)\frac{dt}{t^{1+s/2}}d\mu(x) = \lim_{s\to 0^+}\frac{s}{2}\int_0^1 e^{t\Delta}(\chi_{M\setminus B_r(p)})(p)\frac{dt}{t^{1+s/2}} 
    \end{equation*}
   and exploiting the estimate of Lemma \ref{convergence}. 
    
    \vsp\noindent 
    \textbf{Claim 2.} There holds

\begin{equation}\label{cdscdscd}
     \lim_{s \to 0^+} \frac{s}{2}\int_{B_1(p)} \int_1^\infty H_M(x,p,t)\frac{dt}{t^{1+s/2}} d\mu(x)=0.
\end{equation}
 By the uniform convergence of the heat kernel to zero (in particular, by the result contained in Remark \ref{heatsolconv}) we get that $e^{t\Delta}(\chi_{B_1(p)})(p)\to 0$ as $t\to\infty$. Therefore, for all $\varepsilon>0$ there exists $T=T(\varepsilon)$ such that $e^{t\Delta}(\chi_{B_1(p)})(p)\leq\varepsilon$ for all $t\geq T$, whence
 \begin{equation*}
     \limsup_{s \to 0^+} \frac{s}{2} \int_1^\infty e^{t\Delta}(\chi_{B_1(p)})(p)\frac{dt}{t^{1+s/2}} d\mu(x) \leq \lim_{s\to 0}\frac{s}{2} \int_1^T\frac{dt}{t^{1+s/2}}+ \varepsilon\limsup_{s\to 0}\frac{s}{2}\int_T^\infty\frac{dt}{t^{1+s/2}}\leq\varepsilon \,,
 \end{equation*}
 for all $\varepsilon>0$, proving the second claim.

\vsp
Now, thanks to the first claim we can reduce ourselves to computing
\begin{equation*}
      \lim_{s \to 0^+} \frac{s}{2}\int_{M\setminus B_1(p)} \int_1^\infty H_M(x,p,t)\frac{dt}{t^{1+s/2}} d\mu(x).
\end{equation*}
Then we can then add \eqref{cdscdscd} to the previous limit, which gives zero contribution, and we end up with
\begin{equation*}
    \lim_{s \to 0^+} \frac{s}{2}\int_{M} \int_1^\infty H_M(x,p,t)\frac{dt}{t^{1+s/2}} d\mu(x).
\end{equation*}
Using Fubini and the stochastical completeness of $M$ we get
\begin{equation*}
     \lim_{s \to 0^+} \frac{s}{2}\int_{M} \int_1^\infty H_M(x,p,t)\frac{dt}{t^{1+s/2}} d\mu(x) =  \lim_{s \to 0^+} \frac{s}{2} \int_1^\infty \frac{dt}{t^{1+s/2}} d\mu(x) =1\,,
\end{equation*}
and this concludes the proof. 

\vsp
Conversely assume that \eqref{Stoccompl} holds, then since both the previous claims hold on any connected and geodesically complete Riemannian manifold we have 
\begin{equation*}
    \lim_{s\to 0^+}\frac{s}{2}\int_{M }\int_1^\infty H_M(x,p,t)\frac{dt}{t^{1+s/2}}d\mu(x)=1.
\end{equation*}
Setting $ \mathcal{M} (t,p)=\int_M H_M(x,p,t)d\mu(x)\leq 1$ we can infer that, for every $T>0$
\begin{equation*}
    1=\lim_{s\to 0}\frac{s}{2}\int_T^\infty\frac{\mathcal{M}(t,p)}{t^{1+s/2}}dt\leq \lim_{s\to 0}\frac{s}{2}\int_T^\infty\frac{1}{t^{1+s/2}}dt=1 \,.
\end{equation*}

Now, assume by contradiction that $M$ is not stochastically complete. Then since $\mathcal{M}(t,p)$ is nonincreasing in time and nonnegative, there holds $\lim_{t\to\infty} \mathcal{M}(t,p) \le 1-\delta $ for some $\delta>0$, and we would have $\mathcal{M}(t,p)\leq 1-\delta/2$ for every $t\geq T=T(\delta)$. This gives
\begin{equation*}
    1=\lim_{s\to 0}\frac{s}{2}\int_T^\infty\frac{\mathcal{M}(t,p)}{t^{1+s/2}}dt\leq \lim_{s\to 0}\frac{s}{2}\int_T^\infty\frac{1-\delta/2}{t^{1+s/2}}dt=1-\delta/2 \,,
\end{equation*}
reaching a contradiction, hence $\lim_{t\to\infty}\mathcal{M}(t,p)=1$ and thanks to Lemma \eqref{Mass decay} we conclude.
\end{proof}

\begin{remark}\label{adfs} Following the proof of Proposition \ref{limexistence}, one can see a clear picture of what happens to the limit in $\theta_M(p)$ even when $M$ is not stochastically complete. Indeed, for every Riemannian manifold (not necessarily stochastically complete) and $p\in M$, the limit $\lim_{t\to \infty } \mathcal{M}(t,p)$ exists. This follows from the fact that $\mathcal{M}(\cdot, p)$ is nonincreasing and nonnegative; see Lemma \ref{Mass decay}. Since 
\begin{equation*}
    \mathcal{M}(t,p) = \int_M H_M(p,x,t)\, d\mu(x) = e^{t\Delta} 1
\end{equation*}
is a solution to the heat equation starting from the function equal to one; it follows from the proof above and from standard parabolic estimates that $\mathcal{M}(t,\cdot ) \to \theta_{M}  $ in $C^2_{\rm loc}(M) $ as $t\to \infty$, where $ \theta_{M}: M \to \R $ is a bounded, nonnegative harmonic function on $M$. Therefore:
\begin{itemize}
    \item[(i)] If $M$ is stochastically complete, we have $ \theta_{M} \equiv 1$ (in particular, the value of $\theta_{M}$ does not depend on the point), and the proof above shows $\theta_M=1$. 
    
     \item[(ii)] If $M$ is not stochastically complete but satisfies the $L^\infty-{\rm Liouville}$ property (see Definition \ref{LinfLiouville}) we know that $ \theta_{M} \equiv \theta_\circ \in [0,1) $ and, following the proof of the proposition, one finds that the limit in the definition of $\theta_M$ exists, does not depend on the point $p$ and there holds $\theta_M=\theta_\circ$. Note that such Riemannian manifolds exist and were first constructed in \cite{PinYehu95}. In Example \ref{exampleshit}, we describe one with $\theta_\circ=0$. 
     
     \item[(iii)] If $M$ is not stochastically complete and does not satisfy the $L^\infty-{\rm Liouville}$ property, then in general $ \theta_{M} $ is a nonconstant harmonic function on $M$, and the value of $\theta_M(p)$ can depend on the point $p$. 
\end{itemize}
\end{remark}

Now we are in the position to prove our first main result.
\begin{proof}[Proof of Theorem \ref{mainrandom0}]
    With no loss of generality assume $r<R$. First, we show that the limit does not depend on the radius, that is 
    \begin{equation*}
        \theta_E(p, R) = \theta_E(p,r) \,.
    \end{equation*}
    We have 
    \begin{align*}
    \bigg|  \int_{E \setminus B_R(p)} \mathcal{K}_s(x,p)& d\mu(x) -  \int_{E \setminus B_r(p)} \mathcal{K}_s(x,p) d\mu(x) \bigg| \le  \int_{B_R(p) \setminus B_r(p)} \mathcal{K}_s(x,p) d\mu(x) \\ & \le  C s \int_{B_R(p) \setminus B_r(p)} \int_{0}^1 H_M(x,p,t) \frac{dt}{t^{1+s/2}}d\mu(x) \\ & + C s \int_{B_R(p) \setminus B_r(p)} \int_{1}^{\infty} H_M(x,p,t) \frac{dt}{t^{1+s/2}}d\mu(x)  =: I_1+I_2 \,.
\end{align*} 

For the first integral, by Lemma \ref{convergence} as $s\to 0^+$
\begin{equation*} 
    I_1 \le C s \int_{0}^1 e^{t\Delta}(\chi_{M\setminus B_r(p)})(p) \frac{dt}{t^{1+s/2}}  \le Cs \int_{0}^1 \frac{e^{-c/t}}{t^{1+s}} \, dt \to 0 \,.
\end{equation*}

Regarding the second integral, for all $\ep >0$ by Lemma \ref{gryg} there is $T=T(\ep)>0$ such that $|H_M(x,p,t)|\le \ep$ for all $x\in B_R(p)$ and $t\ge T$, hence 
\begin{align*}
    I_2 & \le Cs \int_{1}^T \int_{B_R(p)} H_M(x,p,t) d\mu(x)\frac{dt}{t^{1+s/2}}+  Cs \int_{T}^\infty \int_{B_R(p)} H_M(x,p,t) d\mu(x) \frac{dt}{t^{1+s/2}} \\&\le Cs\int_{1}^T \frac{dt}{t^{1+s/2}} + Cs \ep  \mu(B_R(p))\int_T^{\infty} \frac{dt}{t^{1+s/2}}\\ & = C (1-T^{-s/2}) +C\ep \mu(B_R(p)) T^{-s/2} \,,
\end{align*}
    letting $s\to 0^+$ (and then $\ep \to 0$) gives $I_2 \to 0$. Hence, taking $s\to 0^+$ shows $\theta_E(p, R) = \theta_E(p,r)$, showing that the limit never depends on the radius. Note that what we have just proved already implies that if $E$ is bounded then the limit exists and $\theta_E=0$, since one can just take $R\gg 1$ so that $E\setminus B_R(p) = \varnothing $. 

    \vsp
    Now fix $q \in M$. For every $p \in B_{1/2}(q)$ we can write
    \begin{equation*}
        \theta_{E}(p)=\lim_{s\to 0^+}\int_{E \setminus B_1(q)} \mathcal{K}_s(x,p)d\mu(x).
    \end{equation*}
    This is possible because we always have independence on the radius. Indeed 
    \begin{equation*}
        \bigg| \int_{E \setminus B_{1/2}(p)}  \mathcal{K}_s(x,p) d\mu(x) -  \int_{E \setminus B_1(q)} \mathcal{K}_s(x,p) d\mu(x) \bigg| \le \int_{B_1(q) \setminus B_{1/2}(p)} \mathcal{K}_s(x,p) \, d\mu(x) \,,
    \end{equation*}
    hence 
    \begin{equation*}
       \limsup_{s\to 0^+} \bigg| \int_{E \setminus B_{1/2}(p)}  \mathcal{K}_s(x,p) d\mu(x) -  \int_{E \setminus B_1(q)} \mathcal{K}_s(x,p) d\mu(x) \bigg| \le \theta_{B_1(q)} =0 \,.
    \end{equation*}
    
   Now set 
    \begin{equation}\label{eq: big theta def}
        \Theta_{E,s}(p) := \frac{s}{2}\int_0^\infty e^{t\Delta}(\chi_{E \setminus B_1(q)})(p)\frac{dt}{t^{1+s/2}} \,,
    \end{equation}
    so that $\theta_E(p) = \lim_{s\to 0^+}  \Theta_{E,s}(p)$.  By Lemma \ref{convergence} we have that that $0 \leq  \Theta_{E,s}(p) \leq C$, for some constant $C>0$ depending only on $M$. Now fix $\varphi\in C_c^\infty(B_{1/2}(q))$, by dominated convergence
     \begin{align}\label{eq: proof main thm 1}
         \int_M\theta_E ( \Delta\varphi ) \, d\mu = \lim_{s\to 0^+}\int_M  \Theta_{E,s} ( \Delta\varphi )  \, d\mu =\lim_{s\to 0^+}\int_M (\Delta  \Theta_{E,s} ) \varphi \, d\mu \,.
     \end{align}
     
     Note that, for fixed $s$ and $p\in B_{1/2}(q)$, we can write
     \begin{equation*}
         \Delta  \Theta_{E,s}(p) =\frac{s}{2}\int_0^\infty \Delta e^{t\Delta}(\chi_{E \setminus B_1(q)})(p)\frac{dt}{t^{1+s/2}}=\frac{s}{2}\int_0^\infty \partial_t e^{t\Delta}(\chi_{E \setminus B_1(q)})(p)\frac{dt}{t^{1+s/2}},
     \end{equation*}
     which, after integration by parts, becomes (note that the boundary term at $t=0^+$ is zero due to Lemma \ref{convergence}) equal to
     \begin{equation*}
         \frac{s}{2} \int_0^\infty e^{t\Delta}(\chi_{E \setminus B_1(q)})(p)\frac{(1+s/2)}{t^{2+s/2}} \, dt \,.
     \end{equation*}
     
     The latter quantity goes to $0$ as $s\to 0^+$, and is uniformly bounded for $s\in (0,1)$, for every $p\in B_{1/2}(q)$. Hence, going back to \eqref{eq: proof main thm 1} we get
     \begin{equation*}
         \int_M\theta_E ( \Delta\varphi ) \, d\mu = 0 \,.
     \end{equation*}
That is, $\theta_E \in L^1_{\rm loc}(M)$ is a very weak solution of $\Delta \theta_E =0$. We're left to prove that $\theta_E$ is smooth and is a classical solution of $\Delta \theta_E=0$. 

In a small chart, in coordinates, one can see that $u$ is (locally) a very weak solution of $\partial_i \big( \sqrt{|{\rm det}(g)|} g^{ij} \partial_j \theta_E \big) =0$. Choosing the chart sufficiently small, we get that the coefficients $\sqrt{|{\rm det}(g)|} g^{ij}$ are smooth and uniformly elliptic. Then, for example by \cite[Theorem 1.3]{reg-very-weak}, we get that $\theta_E \in W^{2,2}_{\rm loc}(M)$ and bootstrapping classical elliptic regularity gives that $\theta_E$ is smooth and harmonic. 

\vsp
     Lastly, \eqref{argwer} follows from the last part of the proof of Proposition \ref{limexistence}, and the fact that $p\mapsto \theta_M(p)$ is harmonic is verbatim the proof we did for $E \subset M$ above.
\end{proof}

Note that, according to Theorem \ref{mainrandom0}, if $M$ possesses the $L^\infty-{\rm Liouville}$ property, then $\theta_E$ is constant for every set $E$ for which it exists. A natural question to ask would be whether some type of converse is true. However, we have not been able to prove or disprove such a statement. We leave this as an open question, and we would be happy to know the answer:

\begin{tcolorbox}[colback=blue!3!white]
   \begin{question} Let $(M,g)$ be a complete Riemannian manifold with $\mu(M)=+\infty$ and with the following property: for every set $E\subset M$ for which $\theta_E$ exists (see \eqref{eq: theta limit def}), $\theta_E$ is constant. 

\vsp \center 
Is it true that $M$ satisfies the $L^\infty-{\rm Liouville}$ property?
\end{question}
\end{tcolorbox}

\vsp
Now we turn to the proof of Theorem \ref{mainrandom1}. To prove this result, we will need Lemma \ref{porcatroia}, which essentially says that for manifolds with $\mu(M)=+\infty$, the singular kernel $\mathcal{K}_s$ locally behaves like that of $\R^n$ as $s\to 0^+$. This is not the case for finite volume manifolds\footnote{Indeed, for finite volume manifolds, the same conclusion \eqref{asrtqw4} holds with constants depending on $s$, but as $s\to 0^+$ the constants do not behave like the ones of $\R^n$.}. Recall the notation of Remark \ref{constantinRn}, where we denote by $\frac{\beta_{n,s}}{|x-y|^{n+s}}$ the singular kernel of $\R^n$ with its standard metric. Note also that $ cs(2-s) \le \beta_{n,s} \le Cs(2-s)$ for some dimensional $c,C>0$. 

\vsp 
The following lemma is a sharpening of \cite[Lemma 2.19]{CFS23} for manifolds with infinite volume. Indeed, in \cite{CFS23}, the authors are not interested in characterizing the sharp dependence from $s$ of $\mathcal{K}_s$ as $s\to 0^+$. Moreover, in \cite{CFS23}, the authors estimate $\mathcal{K}_s$ locally on every complete Riemannian manifold $M$ (both with finite and infinite volume), but the result stated in Lemma \ref{porcatroia} is not true on manifolds with finite volume. 

\begin{lemma}\label{porcatroia}
    Let $(M,g)$ be a complete $n$-dimensional Riemannian manifold with $\mu(M)=+\infty$, and let $p \in M$. Assume that in normal coordinates at $p$ there holds $ \frac{99}{100} |v|^2\le g_{ij}(q)v^iv^j\le \tfrac{101}{100} |v|^2$ and $|\nabla g_{ij}(q)|\le 1/100$ for all $ v \in \R^n$ and $q \in B_1(p)$. Then there exists $\mathcal{K}_s' : B_1(p) \times B_1(p) \to [0,\infty)$ such that 
    \begin{equation*}
        \lim_{s\to 0^+} \sup_{x,y \in B_{1/8}(p) } \left| \mathcal{K}_s(x,y) - \mathcal{K}_s'(x,y) \right|  =0 \,,
    \end{equation*}
    and for all $x,y \in B_{1/8}(p)$
    \begin{equation}\label{asrtqw4}
    c \frac{ \beta_{n,s}}{d(x,y)^{n+s}} \le \mathcal{K}_s'(x,y) \le C \frac{ \beta_{n,s}} {d(x,y)^{n+s}} \,,
\end{equation}
for some dimensional constants $c,C>0$. 
\end{lemma}

We postpone the proof of Lemma \ref{porcatroia} to \autoref{hker est sbs} in the Appendix.

\begin{proof}[Proof of Theorem \ref{mainrandom1}]  
As we can assume $s<s_\circ/2$, it follows from the proof of Proposition \ref{coincideprop2} that the integral in $(-\Delta)^{s/2}_{\rm Si} u$ is absolutely convergent\footnote{Here we are not assuming $M$ being stochastically complete, but in Proposition \ref{coincideprop2} stochastical completeness is only used to have that $(-\Delta)^{s/2}_{\rm B} u = (-\Delta)^{s/2}_{\rm Si} u $ a.e., not to show the absolute convergence of the integrals.} for a.e. $x\in M$, and the principal value is not needed. Moreover, since $u\in H^{s_\circ/2}(M)$ we have 
\begin{equation*}
    \int_{M} (u(x)-u(y))^2 \mathcal{K}_{s_\circ}(x,y) \, d\mu(y) <+\infty 
\end{equation*}
for a.e. $x\in M$. 
Fix $x \in M$ in the intersection of these two sets of full measure, and take $R$\ such that $\supp(u) \subset B_R(x) $. Then 
\begin{align}\nonumber
    (-\Delta)^{s/2}_{\rm Si} u (x) & = \int_M (u(x)-u(y))\mathcal{K}_s(x,y) \, d\mu(y) \\\label{sdafsrf} &= \int_{B_R(x)} (u(x)-u(y))\mathcal{K}_s(x,y) \, d\mu(y) + u(x) \int_{M \setminus B_R(x)} \mathcal{K}_s(x,y) \, d\mu(y) \,.
\end{align}
Note that being $\mu(M)=+\infty$ we have
\begin{equation*}
    \int_{M \setminus B_R(x)} \mathcal{K}_s(x,y) \, d\mu(y) \neq 0 \,.
\end{equation*}
\textbf{Claim.} As $s \to 0^+$ there holds
\begin{equation*}
    \lim_{s\to 0^+} \int_{B_R(x)} (u(x)-u(y))\mathcal{K}_s(x,y) \, d\mu(y) =0\,.
\end{equation*}
Indeed, let $\rho \ll 1$ small that will be chosen later. We denote here by $C$ a constant which does not depend on $s$. Then 
\begin{align*}
   \bigg| \int_{B_R(x)} & (u(x)-u(y)) \mathcal{K}_s(x,y) \, d\mu(y) \bigg| \\ & =  \bigg| \int_{B_{\rho}(x)} (u(x)-u(y))\mathcal{K}_s(x,y) \, d\mu(y) + \int_{B_R(x) \setminus B_{\rho}(x) } (u(x)-u(y))\mathcal{K}_s(x,y) \, d\mu(y) \bigg| \\ & \le  \int_{ B_\rho(x)} |u(x)-u(y)| \mathcal{K}_s(x,y) \, d\mu(y) + 2 \| u\|_{L^\infty} \int_{B_R(x) \setminus B_{\rho}(x)} \mathcal{K}_s(x,y) \, d\mu(y) \,.
\end{align*}
We estimate these two integrals separately. Let $\mathcal{K}_s'$ be the singular kernel given by Lemma \ref{porcatroia}, applied with $\rho$ sufficiently small and suitably rescaled. For the first integral, Lemma \ref{porcatroia} gives
\begin{equation}\label{aethjdty}
    \limsup_{s\to 0^+} \int_{B_\rho(x)} |u(x)-u(y)| \big( \mathcal{K}_s(x,y) - \mathcal{K}_s'(x,y) \big) \, d\mu(y) = 0 \,.
\end{equation}
Moreover, by the bounds of Lemma \ref{porcatroia} and since $u\in H^{s_\circ/2}(M)$, for a.e. $x\in M$
\begin{equation*}
    \int_{B_\rho(x)} \frac{(u(x)-u(y))^2}{d(x,y)^{n+s_\circ} } \, dy \le C(s_\circ) \int_{B_\rho(x)} (u(x)-u(y))^2 \mathcal{K}_{s_\circ}(x,y) \, dy <+\infty.
\end{equation*}
Hence, by Lemma \ref{porcatroia} again and Holder's inequality
\begin{align*}
     \int_{B_\rho(x)} & |u(x)-u(y)|\mathcal{K}'_s(x,y) \, d\mu(y) \le Cs \int_{B_\rho(x)} \frac{|u(x)-u(y)|}{d(x,y)^{n+s}} \, d\mu(y) \\ & \le Cs \left( \int_{B_\rho(x)} \frac{(u(x)-u(y))^2}{d(x,y)^{n+s_\circ} } \, dy \right)^{1/2} \left( \int_{B_\rho(x)} \frac{1}{d(x,y)^{n+2s-s_\circ} dy} \right)^{1/2} \\ & \le Cs \left( \frac{\rho^{s_\circ-2s}}{s_\circ-2s} \right)^{1/2} \to 0 \,,
\end{align*}
as $s\to 0^+$, where in the second-last inequality we have used polar coordinates for $\rho$ sufficiently small (possibly depending on $x$). Thus, with \eqref{aethjdty} we have that the first integral tends to zero. 

\vsp
Regarding the second integral, one can note that we have proved in part $(i)$ of Theorem \ref{mainrandom0} that, for every $x\in M$ and $r,R>0$
\begin{equation*}
   \lim_{s\to 0^+} \int_{B_R(x) \setminus B_r(x)} \mathcal{K}_s(x,y) d\mu(y) =0\,,
\end{equation*}
since $B_R(x)$ is a bounded set, and this concludes the proof of the claim.

\vsp
 Moreover, by the very definition of $\theta_M$ we have
\begin{align}\label{thetaMeqn}
    \lim_{s\to 0^+}  \int_{M \setminus B_R(x)} \mathcal{K}_s(x,y) \, d\mu(y)  =\theta_M(x) \,, 
\end{align}
hence letting $s\to 0^+$ in \eqref{sdafsrf} gives
\begin{equation*}
    \lim_{s\to 0^+} (-\Delta)^{s/2}_{\rm Si} u (x) = \theta_M(x)  u(x) \,,
\end{equation*}
for a.e. $x\in M$, and this concludes the proof. 
\end{proof}

To prove our result Theorem \ref{maincompact} on the asymptotics for infinite volume, one needs also to know the asymptotics as $s\to 0^+$ of the fractional $s$-perimeter on the entire $M$, that is when $\Omega \equiv M$. This is addressed by Theorem \ref{globlimnoncompact} below on the asymptotics of the fractional Sobolev seminorms. This result is the counterpart of Theorem \ref{globlimcomp} in the case of infinite volume. 

\begin{theorem}\label{globlimnoncompact}  Let $(M,g)$ be a complete Riemannian manifold with $\mu(M)=+\infty$, and let $s_\circ\in (0,1)$. Then, for every $u \in H^{s_\circ/2}(M) \cap L^\infty(M)$ with bounded support there holds
    \begin{equation*}
         \lim_{s \to 0^+} \frac{1}{2} [u]_{H^{s/2}(M)}^2 =  \int_M u^2 \theta_M \, d\mu \,.
     \end{equation*}
\end{theorem}
\begin{proof}
    Formally, one would like to infer that 
    \begin{align*}
        \frac{1}{2} [u]_{H^{s/2}(M)}^2 & := \frac{1}{2} \iint_{M\times M} (u(x)-u(y))^2 \mathcal{K}_{s}(x,y) \, d\mu(x) d\mu(y) \\ & = \int_M u (-\Delta)^{s/2}_{\rm Si} u \, d\mu \conv{s\to 0^+} \int_M u^2 \theta_M \, d\mu \,,
    \end{align*}
    where the first equality is the very definition of the seminorm. The second inequality is nontrivial since the integrals one would write in the few lines of a proof are not absolutely convergent in general. Moreover, for the last step of taking the limit as $s\to 0^+$ one needs to show that the a.e. convergence $(-\Delta)^{s/2}_{\rm Si} u \to \theta_M u $ of Theorem \ref{mainrandom1} can be upgraded to weak convergence in $L^2(M)$. Now we shall justify both steps.
    
\vsp
\textbf{Step 1.} We have 
\begin{equation}\label{vbnvbn}
    \frac{1}{2} \iint_{M\times M} (u(x)-u(y))^2 \mathcal{K}_{s}(x,y) \, d\mu(x) d\mu(y) = \int_M u (-\Delta)^{s/2}_{\rm Si} u \, d\mu \,.
\end{equation}
    Fix $\ep>0$ and let 
    \begin{equation*}
        (-\Delta)^{s/2}_\ep u (x) := \int_{M\setminus B_\ep(x)} (u(x)-u(y))\mathcal{K}_s(x,y) \, d\mu(y) \,.
    \end{equation*}
Let also $D:=\{(z,z) \, : \, z\in M \} $ denote the diagonal of $M\times M$ and $D_\delta$ a $\delta$-neighborhood of $D$. We have 
\begin{align*}
     \iint_{M\times M \setminus D_{\ep/\sqrt{2}}}   (u(x)-u(y))^2 &\mathcal{K}_{s}(x,y) \, d\mu(x) d\mu(y) \\ & = \iint_{M\times M \setminus D_{\ep/\sqrt{2}}} u(x)(u(x)-u(y))\mathcal{K}_{s}(x,y) \, d\mu(x) d\mu(y) \\ & \s \s - \iint_{M\times M \setminus D_{\ep/\sqrt{2}}} u(y)(u(x)-u(y))\mathcal{K}_{s}(x,y) \, d\mu(x) d\mu(y) \\ &= 2 \iint_{M\times M \setminus D_{\ep/\sqrt{2}}} u(x)(u(x)-u(y))\mathcal{K}_{s}(x,y) \, d\mu(x) d\mu(y) \\ &= 2 \int_M \int_{M\setminus B_{\ep}(x)} u(x)(u(x)-u(y))\mathcal{K}_{s}(x,y) \, d\mu(y) d\mu(x) \\ & = 2\int_{M} u (-\Delta)^{s/2}_{\ep} u \, d\mu \,,
\end{align*}
where splitting the integral and Fubini are justified since the integrals are absolutely convergent. Indeed 
    \begin{align*}
       & \int_M \int_{M\setminus B_{\ep}(x)} |u(x)(u(x)-u(y))|\mathcal{K}_{s}(x,y) \, d\mu(y) d\mu(x) \\ & \le \int_M |u(x)|^2 \int_{M\setminus B_{\ep}(x)} \mathcal{K}_{s}(x,y) \, d\mu(y) d\mu(x) + \int_M |u(x)| \int_{M\setminus B_{\ep}(x)} |u(y)|\mathcal{K}_{s}(x,y) \, d\mu(y) d\mu(x) \,,
    \end{align*}
    but by Lemma \ref{convergence} 
    \begin{align*}
        \int_{M\setminus B_{\ep}(x)} \mathcal{K}_{s}(x,y) \, d\mu(y) &= C \int_{0}^\infty \left( \int_{M\setminus B_{\ep}(x)} H_M(x,y,t) \, d\mu(y) \right)\frac{dt}{t^{1+s/2}} \\ & \le C \int_{0}^\infty \frac{e^{-c/t}}{t^{1+s/2}} dt \le C \,,
    \end{align*}
    for some $C$ depending on $s$ and $\ep$. Hence 
    \begin{equation*}
        \int_M \int_{M\setminus B_{\ep}(x)} |u(x)(u(x)-u(y))|\mathcal{K}_{s}(x,y) \, d\mu(y) d\mu(x) \le C(\|u \|_{L^\infty}, \mu(\supp(u)), \ep, s ) <+\infty \,,
    \end{equation*}
    and this shows the absolute convergence.

    \vsp
    Moreover, by Proposition \ref{coincideprop2} for a.e. $x\in M$ the integral in $(-\Delta)^{s/2}_{\rm Si} u$ is absolutely convergent, then 
    \begin{align*}
        \int_M \big|(-\Delta)^{s/2}_{\rm Si} u-(-\Delta)^{s/2}_{\ep} u \big|^2 d\mu \le \int_M \left| \int_{B_\ep(x)} |u(x)-u(y)|\mathcal{K}_s(x,y) d\mu(y) \right|^2 d\mu(x) \,,
    \end{align*}
    and the right hand side ternds to $0$ as $\ep \to 0$. Indeed, as $\ep\to 0 $, by the very same argument at the end of the proof of Theorem \ref{mainrandom1} there holds
    \begin{equation*}
        \int_{B_\ep(x)} |u(x)-u(y)|\mathcal{K}_s(x,y) d\mu(y) \to 0 \,,
    \end{equation*}
    for a.e. $x\in M$, and for $x$ fixed the convergence is monotone (decreasing) since the integrand is positive. Hence we have proved $(-\Delta)^{s/2}_{\ep} u \to (-\Delta)^{s/2}_{\rm Si} u $ in $L^2(M)$ as $\ep \to 0$. Now, letting $\ep \to 0$ in 
\begin{equation*}
   \frac{1}{2} \iint_{M\times M \setminus D_{\ep/\sqrt{2}}}   (u(x)-u(y))^2 \mathcal{K}_{s}(x,y) \, d\mu(x) d\mu(y) = \int_{M} u (-\Delta)^{s/2}_{\ep} u \, d\mu \,,
\end{equation*}
together with the monotone convergence theorem on the left-hand side, we get the equality of the seminorms and this completes the proof of Step 1. 
    
\vsp
\textbf{Step 2.} There holds
\begin{equation*}
    (-\Delta)^{s/2}_{\rm Si} u \rightharpoonup \theta_M u \s \textnormal{weakly in } L^2(M) \,.
\end{equation*}

The convergence a.e. is given by Theorem \ref{mainrandom1}. To prove that the convergence holds weakly in $L^2(M)$, we show that $(-\Delta)^{s/2}_{\rm Si} u$ is equibounded in $L^2(M)$. By \eqref{uqweriufgweur} there is $C$ depending only on $s_\circ$ such that
\begin{align*}
    \| (-\Delta)^{s/2}_{\rm Si} u \|_{L^2(M)}^2 \le  C \|u\|_{L^2(M)}^2 +  Cs^2\| u \|_{H^{s_\circ}(M)}^2 \,, 
\end{align*}
and hence
\begin{equation*}
    \limsup_{s\to 0^+} \| (-\Delta)^{s/2}_{\rm Si} u \|_{L^2(M)}^2 \le C \|u\|_{L^2(M)}^2 <+\infty \,.
\end{equation*}
This concludes Step 2 and, sending $s\to 0^+$ in \eqref{vbnvbn} concludes the proof.
\end{proof}
\begin{remark} Note that the equivalence of the seminorms \eqref{vbnvbn} always holds for characteristic functions, without any assumption. Indeed for every measurable $E\subset M$ 
    \begin{align*}
     2\int_M \chi_E \cdot  (-\Delta)^{s/2}_{\rm Si} \chi_E \, dx  & = 2\int_E \left( \lim_{\ep\to 0} \int_{M\setminus B_\ep(x)} (1-\chi_E(y)) \mathcal{K}_s(x,y) dy\right)   dx \\ &=  2\int_E \left( \lim_{\ep\to 0} \int_{(M\setminus B_\ep(x))\cap E^c} \mathcal{K}_s(x,y) dy\right)   dx \\ & = 2\int_E \int_{E^c} \mathcal{K}_s(x,y) dy =  [\chi_E]^2_{H^{s/2}(M)}  \,,
\end{align*}
where the second-last equality follows by the monotone convergence theorem.
\end{remark}

\begin{proof}[Proof of Theorem \ref{mainrandom2}]
     Since $M$ is stochastically complete, by Proposition \ref{coincideprop4} we have $H^{s_\circ/2}(M) \subset {\rm Dom}((-\Delta)^{s/2}_{\rm Spec})$. The equality a.e. of the fractional Laplacian, then follows by Proposition \ref{coincideprop2} and Proposition \ref{coincideprop1}. 

     \vsp
     To prove \eqref{rand1}, \eqref{rand2} one can argue similarly to the proof of Theorem \ref{globlimcomp}. Indeed, as $s\to 0^+$ for every $v\in L^2(M)$ we have
         \begin{align*}
             \lp(-\Delta)^{s/2}_{\rm Spec} u, v \rp = \int_{\sigma(-\Delta)} \lambda^{s/2} d \lp E_\lambda u , v  \rp \to \int_{\sigma(-\Delta)\setminus \{0\}} d \lp E_\lambda u , v  \rp = \lp u,v \rp - \lp E_0 u,v \rp \,,
         \end{align*}
         where $E_0$ is the projector onto the eigenspace of $-\Delta$ relative to the eigenvalue $\lambda=0$. By Theorem \ref{dfsdfs} every $L^2(M)$ harmonic function is constant, hence we have two cases: 
        \begin{itemize}
            \item[$(i)$] If $\mu(M)<+\infty$ then the eigenspace of $\lambda=0$ is the span of the eigenfunction $\mu(M)^{-1/2}$, then $E_0u=\frac{1}{\mu(M)}\int_M u \, d\mu$ and this gives \eqref{rand1}. 
             \item[$(ii)$] If $\mu(M)=+\infty$ then $E_0u=0$ and we have \eqref{rand2}. 
        \end{itemize}
         This concludes the proof.
 \end{proof}
 
\begin{remark}
     When $M$ is stochastically complete with $\mu(M)=+\infty$ the convergence in \eqref{rand2} also follows by Theorem \ref{mainrandom1}, since $\theta_M\equiv 1$ in this case. Nevertheless, the argument carried on in Theorem \ref{mainrandom1} is much more general and shows what happens in the limit on any manifold with $\mu(M)=+\infty$, even when $M$ is not stochastically complete (i.e. when $(-\Delta)_{\rm Si}^{s/2}$ and $(-\Delta)_{\rm Spec}^{s/2}$ do not coincide). 
 \end{remark}

 \section{Asymptotics: finite volume manifolds}\label{section4}

 \subsection{Global asymptotics}
 
We first give a simple proof of Theorem \ref{maincompact} in the case $\Omega =M$, using our results from \autoref{flapsection} on the equivalence of the spectral fractional Laplacian and ours defined by the singular integral \eqref{flap}.

 \begin{theorem}\label{globlimcomp}
     Let $(M,g)$ be a complete Riemannian manifold with $\mu(M)<+\infty$ and let $s_\circ \in (0,1)$. Then, for every $u \in H^{s_\circ/2}(M)$ there holds
     \begin{equation*}
         \lim_{s \to 0^+} \frac{1}{2} [u]_{H^{s/2}(M)}^2 = \|u \|_{L^2(M)}^2 - \frac{1}{\mu(M)} \left( \int_M u \, d\mu\right)^2.
     \end{equation*}
 \end{theorem}
%  \begin{proof}[Proof - closed manifolds]
% Let $\{\phi_k\}_{k=1}^{\infty} $ be an $L^2(M)$ orthonormal basis of eigenfunctions for $(-\Delta) $, with eigenvalues $0 \le \lambda_1 \le \lambda_2 \le \dotsc \le \lambda_k \to \infty$ as $k \to \infty$. Since $M$ is compact, we have that $\lambda_1=0$ and the first eigenvalue $\lambda_1$ is simple with eigenfunction $\phi_1 \equiv \mu(M)^{-1/2}$. Then, since $u\in L^2(M)$ we can write
% \begin{equation*}
%     u= \sum_{k=1}^{\infty} u_k \phi_k \,,
% \end{equation*}
% where $u_k = \lp u, \phi_k\rp_{L^2(M)} $ and the limit has to be understood in $L^2(M)$. Then 
% \begin{equation*}
%     (-\Delta)^{s/2} u = \sum_{k=1}^\infty \lambda_k^{s/2} u_k \phi_k \,,
% \end{equation*}
% where 
% \begin{equation*}
%     (-\Delta)^{s/2} u (x) = \int_M (u(x)-u(y)) \mathcal{K}_s(x,y) \, d\mu(y)
% \end{equation*}
% is the fractiona $(s/2)$-Laplacian on $M$ and $\mathcal{K}_s$ is defined as in \eqref{singkerdef}. Hence, for all $s<s_\circ$ 
% \begin{equation*}
%      \frac{1}{2} [u]_{H^{s/2}(M)}^2 =  \int_{M} u (-\Delta)^{s/2} u \, d\mu = \sum_{k=2}^{\infty} \lambda_k^{s/2} u_k^2 \,.
% \end{equation*}
% Taking the limit as $s\to 0^+$ gives
% \begin{equation*}
%     \lim_{s \to 0^+} \frac{1}{2} [u]_{H^{s/2}(M)}^2  = \sum_{k=2}^{\infty} u_k^2 = \|u \|_{L^2(M)}^2 - u_1^2 = \|u \|_{L^2(M)}^2 - \frac{1}{\mu(M)} \left( \int_M u \, d\mu\right)^2,
% \end{equation*}
% as desired.
% \end{proof}

 \begin{proof}[Proof]
Let $ \{E_\lambda \}_{\lambda \ge 0} $ be the spectral resolution of the Laplacian $-\Delta$ on $L^2(M)$, and let $\sigma(-\Delta) \subset [0,\infty)$ be the spectrum of $-\Delta$. In particular, for every $u\in L^2(M)$, $d\lp E_\lambda u,u \rp $ is a regular Borel (real valued) measure on $[0,\infty)$ concentrated on $\sigma(-\Delta)$, and with
\begin{equation*}
    \| u\|^2_{L^2(M)} = \int_{\sigma(-\Delta)} d\lp E_\lambda u,u \rp \,.
\end{equation*}

We refer to \cite[Appendix A.5]{GrygBook} for an introduction and properties of the spectral resolution. Since $\mu(M)<+\infty$, we have that $0 \in \sigma(-\Delta)$ lies in the point spectrum with eigenfunction $\phi_0 = \mu(M)^{-1/2}$. Then
\begin{equation*}
    -\Delta  = \int_{\sigma(-\Delta)} \lambda dE_\lambda  \,, \s \mbox{and  } \s   (-\Delta)^{s/2}_{\rm Spec}  = \int_{\sigma(-\Delta)} \lambda^{s/2} dE_\lambda \,,  
\end{equation*}
on ${\rm Dom}((-\Delta)^{s/2}_{\rm Spec}) := \big\{u \in L^2(M) \, : \, \int_{\sigma(-\Delta)} \lambda^{s} \, d\lp E_\lambda u, u \rp <+\infty \big\}$. 

\vsp 
Hence, for all $s<s_\circ$ by Corollary \ref{corolequivalence}
\begin{equation*}
     \frac{1}{2} [u]_{H^{s/2}(M)}^2 =  \int_{M} u (-\Delta)_{\rm Si}^{s/2} u \, d\mu = \int_{\sigma(-\Delta)\setminus \{ 0\}} \lambda^{s/2} d\lp E_\lambda u,u \rp \,.
\end{equation*}

Taking the limit as $s\to 0^+$ gives
\begin{equation*}
    \lim_{s \to 0^+} \frac{1}{2} [u]_{H^{s/2}(M)}^2  = \int_{\sigma(-\Delta)\setminus \{ 0\}} d\lp E_\lambda u,u \rp = \|u \|_{L^2(M)}^2 - \lp E_0 u, u \rp = \|u \|_{L^2(M)}^2 - \frac{1}{\mu(M)} \left( \int_M u \, d\mu\right)^2,
\end{equation*}
where in the last line we have used that $E_0$ is the projector onto the eigenspace of $-\Delta$ relative to the eigenvalue $\lambda=0$, but by a result of Yau (see Theorem \ref{dfsdfs}) on a complete manifold every $L^2(M)$ harmonic function is constant and then $\lp E_0 u, u \rp = \lp  \phi_0,u \rp_{L^2(M)}^2 = \frac{1}{\mu(M)} \left( \int_M u \, d\mu\right)^2 $.
\end{proof}

 \begin{remark}
This result allows us to prove our main theorem in the case $\Omega=M$. Indeed, if $E\subset M$ is such that $P_{s_\circ}(E)<+\infty$ for some $s_\circ \in (0,1)$, then taking $u = \chi_E$ in Theorem \ref{globlimcomp} gives
\begin{equation*}
    \lim_{s\to 0^+} \frac{1}{2} P_s(E)= \mu(E)-\frac{1}{\mu(M)} \mu(E)^2 = \frac{\mu(E)\mu(E^c)}{\mu(M)} \,.
\end{equation*}
\end{remark}

 \subsection{Localized asymptotics and proof of Theorem \ref{maincompact}}
Now we turn to the proof of the main result on the asymptotics for finite volume Theorem \ref{maincompact}

\begin{lemma}\label{ex 27}
Let $(M,g)$ be a complete Riemannian manifold, and let $A,B \subset M$ two disjoint measurable sets with (say) $\mu(A)<+\infty$. If $\J_{s_\circ}(A,B) < +\infty$ for some $s_\circ \in (0,1)$ then 
\begin{equation*}
       \lim_{s\to 0^+} \left| \J_s(A,B) - \frac{1}{|\Gamma(-s/2)|} \iint_{A\times B} \int_{1/s}^{\infty} H_M(x,y,t)\frac{dt}{t^{1+s/2}} \, d\mu(x) d\mu(y) \right|  = 0 \,.
    \end{equation*}
\end{lemma}
\begin{proof}
   Since $\int_M H_M(x,y,t) \, d\mu(x) \le  1 $ for all $y\in M$ and $t\in (0,\infty)$ we have
\begin{align*}
   & \bigg| \J_s(A,B) - \frac{1}{|\Gamma(-s/2)|} \iint_{A\times B} \int_{1/s}^{\infty} H_M(x,y,t)\frac{dt}{t^{1+s/2}} \, d\mu(x) d\mu(y) \bigg| \\ & =  \iint_{A\times B} \left( \mathcal{K}_s(x,y)- \frac{1}{|\Gamma(-s/2)|} \int_{1/s}^{\infty} H_M(x,y,t)\frac{dt}{t^{1+s/2}}  \right) \, d\mu(x) d\mu(y) \\  & =  \frac{1}{|\Gamma(-s/2)|} \iint_{A\times B} \left(\int_{0}^{1} H_M(x,y,t)\frac{dt}{t^{1+s/2}}+ \int_{1}^{1/s} H_M(x,y,t)\frac{dt}{t^{1+s/2}} \right) d\mu(x) d\mu(y) \\ & =  \frac{1}{|\Gamma(-s/2)|} \left( \iint_{A\times B} \int_{0}^{1} H_M(x,y,t)\frac{dt}{t^{1+s/2}} d\mu(x)d\mu(y) + \int_A \int_{1}^{1/s} \left(\int_B H_M(x,y,t) \, d\mu(x) \right)\frac{dt}{t^{1+s/2}} d\mu(y) \right) \\ & \le C s \iint_{A\times B} \int_{0}^{\infty} H_M(x,y,t)\frac{dt}{t^{1+s_\circ/2}} + Cs \mu(A) \int_{1}^{1/s} \frac{dt}{t^{1+s/2}} \\[4pt] & = C s \J_{s_\circ}(A,B) + C\mu(A)(1-s^{s/2}) \,,
\end{align*}
and taking $s\to 0^+$ concludes the proof. 
\end{proof}

\begin{proof}[Proof of Theorem \ref{maincompact}] First, we claim that
\begin{equation}\label{asd2}
    \lim_{s\to 0^+} \frac{1}{|\Gamma(-s/2)|} \iint_{A\times B} \int_{1/s}^{\infty} H_M(x,y,t)\frac{dt}{t^{1+s/2}} \, d\mu(x) d\mu(y) = \frac{\mu(A)\mu(B)}{\mu(M)} \,.
\end{equation}

Indeed
\begin{equation*}
    s \int_{1/s}^{\infty} H_M(x,y,t)\frac{dt}{t^{1+s/2}} = s^{1+s/2} \int_{1}^{\infty} H(x,y,r/s) \frac{dr}{r^{1+s/2}} \,,
\end{equation*}
and since by Lemma \ref{gryg}  as $t\to +\infty$ the heat kernel $H_M(x,y,t)$ converges to $1/\mu(M)$ for all $x,y \in M$, we get 
\begin{align*}
    \lim_{s\to 0^+} \frac{s}{2} \iint_{A\times B} \int_{1/s}^{\infty} H_M(x,y,t)\frac{dt}{t^{1+s/2}} \, d\mu(x) d\mu(y)  &= \frac{\mu(A)\mu(B)}{\mu(M)} \lim_{s\to 0^+}  (s/2)\, s^{s/2} \int_{1}^{\infty} \frac{dr}{r^{1+s/2}} \\ & = \frac{\mu(A)\mu(B)}{\mu(M)} \,.
\end{align*}
Then, putting together Lemma \ref{ex 27} and \eqref{asd2} readily implies
\begin{equation*}
    \lim_{s\to 0^+} \J_s(A,B) = \frac{\mu(A)\mu(B)}{\mu(M)} \,.
\end{equation*}

Lastly, since $P_{s_\circ}(E, \Omega) <+\infty$ and
\begin{equation*}
     \frac{1}{2} P_s(E,\Omega) = \J_s(E\cap \Omega, E^c \cap \Omega ) + \J_s(E\cap \Omega, E^c \cap \Omega^c ) + \J_s(E\cap \Omega^c, E^c \cap \Omega ) \,,
\end{equation*}
the theorem follows by letting $s\to 0^+$.

\end{proof}

	In \cite{CaCiLaPa21a} the authors prove the following result regarding the $s$-perimeter of the Gaussian space. Since the total mass of the Gaussian space is one, we see that this is formally identical to our Theorem \ref{maincompact} for finite volume.

 \begin{theorem}[Main Theorem in \cite{CaCiLaPa21a}] Let $\Omega \subset \R^n$ be an open and connected set with Lipschitz boundary. Then, for any $E\subset \R^n$ measurable set such that $P_{s_\circ}^\gamma(E, \Omega) <+\infty$ for some $s_\circ \in (0,1)$ there holds
 \begin{equation*}
     \lim_{s\to 0^+} \frac{s}{2}P_s^\gamma (E; \Omega) = \gamma(E)\gamma(E^c \cap \Omega) + \gamma(E\cap \Omega)\gamma(E^c \cap \Omega^c) \,,
 \end{equation*}
 where $P_s^\gamma(E, \Omega)$ is the fractional Gaussian perimeter
 \begin{align*}
     & P_s^\gamma  (E, \Omega) \\ & = \iint_{E\cap \Omega \times E^c \cap \Omega} \mathcal{K}_s(x,y) \, d\gamma_x d\gamma_y + \iint_{E\cap \Omega \times E^c \cap \Omega} \mathcal{K}_s(x,y) \, d\gamma_x d\gamma_y + \iint_{E\cap \Omega \times E^c \cap \Omega} \mathcal{K}_s(x,y) \, d\gamma_x d\gamma_y \,,
 \end{align*}
 and $\mathcal{K}_s(x,y)$ is defined as in \eqref{singkerdef} with on the right-hand side the heat kernel $H_{\gamma}$ of the Gaussian space $(\R^n, \gamma)$, where $d\gamma(x) = \frac{1}{(2\pi)^{n/2}} e^{-|x|^2/2} \mathcal{L}^n(dx)$. 
     
 \end{theorem}

The proof in \cite{CaCiLaPa21a} follows the same lines as our proof of Theorem \ref{maincompact}, but the authors heavily use the fact that they know the explicit form of the heat kernel $H_\gamma$ for the Gaussian space. In the next subsection, we briefly explain how our method implies their result when applied to weighted manifolds.

\subsection{Weighted manifolds}\label{Weighted manifolds} Our result for finite volume manifolds extends, with proofs mutatis mutandis, to the case of weighted manifolds with finite volume, implying the one in \cite{CaCiLaPa21a}.

\vsp
A weighted manifold is a Riemannian manifold $(M,g)$ endowed with a measure $\mu$ that has a smooth positive density with respect to the Riemannian volume form $dV_g$. The space $(M,g,\mu)$ features the so-called weighted Laplace operator $-\Delta_\mu$, generalizing the Laplace-Beltrami operator, which is symmetric with respect to measure $\mu$. It is possible to extend $-\Delta_\mu$ to a self-adjoint operator in $L^2(M, \mu)$, which allows one to define the heat semigroup $e^{t\Delta_\mu}$ as one would on a classical Riemannian manifold. The heat semigroup has the integral kernel $H_\mu(x,y,t)$, which is called the heat kernel of $(M,g,\mu)$, and has completely analogous properties as the classical one. For every detail regarding the heat kernel on weighted manifolds, we refer to the survey \cite{grygheatweighted}. 

\vsp
In this case, we see that our proof applies since Lemma \ref{gryg} also holds (with the same proof) on geodesically complete weighted manifolds, and also Theorem \ref{globlimcomp} holds with the same proof, since our results from \autoref{flapsection} are valid for weighted manifolds too.

\vsp
     Moreover, our method works also for manifolds with boundary and finite volume. Indeed, if $(M,g)$ is a complete manifold with (possibly empty) boundary and finite volume, and one defines $\mathcal{K}_s(x,y)$ by \eqref{singkerdef} with the heat kernel with Neumann boundary conditions on the right-hand side, then the same proof applies.

 \section{Asymptotics: infinite volume manifolds}\label{section5}

\subsection{Global asymptotics}

\begin{corollary}\label{sdfgbhsdt}
   Let $(M,g)$ be stochastically complete and with $\mu(M)=+\infty$. Let $E\subset M$ be bounded and such that $P_{s_\circ}(E)<+\infty$ for some $s_\circ \in (0,1)$. Then
    \begin{equation*}
        \lim_{s\to 0^+} \frac{1}{2} P_s(E) = \mu(E) \,.
    \end{equation*}
\end{corollary}
\begin{proof}
    Since $M$ is stochastically complete, by Proposition \ref{limexistence} we have $\theta_M\equiv 1$. Then the result follows taking $u = \chi_E$ in Theorem \ref{globlimnoncompact}. 
\end{proof}

One can note that stochastical completeness is not really needed in Corollary \ref{sdfgbhsdt}. Even when $M$ is not stochastically complete, by Theorem \ref{mainrandom0}, we know that $\theta_M$ is a (possibly non-constant) bounded harmonic function with values in $[0,1]$. Then, by Theorem \ref{globlimnoncompact} with $u=\chi_E$ again

 \begin{equation*}
        \lim_{s\to 0^+} \frac{1}{2} P_s(E) = \int_E \theta_M \, d\mu 
    \end{equation*}
    
Consequently, if in particular $\theta_M \equiv \theta_\circ \in [0,1] $ we have
     \begin{equation}\label{jaja}
        \lim_{s\to 0^+} \frac{1}{2} P_s(E) = \theta_\circ \mu(E) \,,
    \end{equation}
    for every $E$ bounded with $P_{s_\circ}(E)<+\infty$. This feature led us to note the following example, which shows that, interestingly enough, Riemannian manifolds with $\theta_M \equiv \theta_\circ=0$ exists.

\begin{example}\label{exampleshit}
    There exists a complete Riemannian manifold $N$ where the asymptotics of the fractional $s$-perimeter as $s\to 0^+$ is zero for every set, that is: for every bounded $E$ with $P_{s_\circ}(E)<+\infty $  for some $s_\circ\in (0,1)$ there holds 
    \begin{equation*}
        \lim_{s\to 0^+} P_s(E) = 0 \,.
    \end{equation*}

    By \eqref{jaja} above we see that it is enough to provide an example of a Riemannian manifold $N$ with $\theta_N(p) \equiv  \theta_\circ =0$, meaning that the limit does not depend on the point $p$ and is always zero. Moreover, by part $(ii)$ of Remark \ref{adfs} this is satisfied if $N$ has the $L^\infty-{\rm Liouville}$ property, is not stochastically complete and
    \begin{equation*}
        \mathcal{N}(t,p)=\int_N H_N(x,p,t) \, d\mu(x)  \to 0 \,,\s \textnormal{as} \,\, t\to \infty \,.
     \end{equation*}
     
A complete Riemannian manifold $N$ with these properties actually exists, and we now sketch how it is constructed. We want $N$ such that
\begin{itemize}
\setlength\itemsep{3pt}
    \item[$(i)$] $N$ has the $L^\infty-{\rm Liouville}$ property.
    \item[$(ii)$] $N$ is not stochastically complete.
    \item[$(iii)$] For every $p\in N$ we have $\mathcal{N}(t,p)=\int_N H_N(x,p,t) \, d\mu(x)  \to 0 $.
\end{itemize}

The construction of $N$ that satisfies $(i), (ii)$ is taken from \cite[Section 13.5]{Gryg22222}, which in turn builds on the first such example found by Pinchover in \cite{PinYehu95}. Here, we note that it satisfies also $(iii)$. 

\begin{figure}[h]
\caption{The two dimensional jungle-gym in $\R^3$. Picture taken from \cite{Gryg22222}.}
\centering
\includegraphics[width=0.5\textwidth]{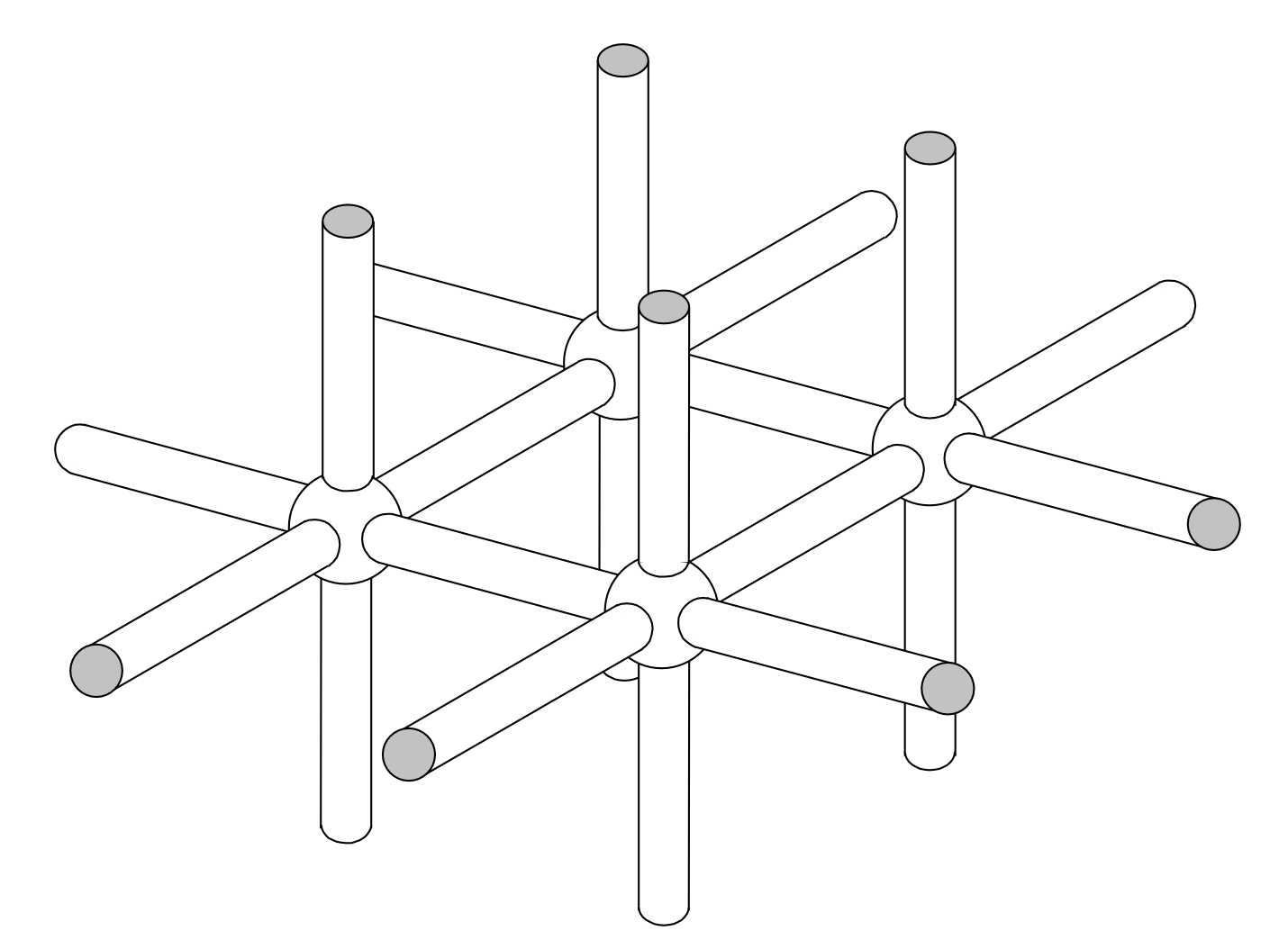}
\label{figA}
\end{figure}

\vsp
    Start from the two-dimensional jungle-gym $JG^2$ in $\R^3$ as in Figure \ref{figA}. This is done by smoothly connecting the lattice $\mathbb{Z}^3 \subset \R^3$ with necks. Let $g$ be the standard metric on $JG^2$ induced by the embedding in $\R^3$. Fix $o\in JG^2$ and let $r:=d(o,x)$. One can show that $JG^2$ has the $L^\infty-{\rm Liouville}$ property. Moreover, there holds $\mu(B_R(o)) \le CR^3$, and the Green function grows at most as $G(o,x) \le C/r$ for large $r$. Let $\rho:JG^2 \to [0,+\infty) $ be a smooth positive function with $\rho=1$ in $[0,1]$ and $\rho(r) \sim \frac{1}{r\log(r)}$ for large $r$, and consider the conformal metric $\widehat{g} := \rho^2(r) g $ on $JG^2$. We claim that $N:=(JG^2, \widehat{g})$ has the desired properties. Since 
\begin{equation*}
    \int_1^\infty \rho(r)dr =\infty \,,
\end{equation*}
then $N$ is geodesically complete and hence complete. Moreover, as the Laplacian is conformally invariant in dimension two, $JG^2$ with its standard metric and $N$ have the same harmonic functions, and thus $N$ also has the $L^\infty-{\rm Liouville}$ property and satisfies $(i)$. Denote by $\widehat{G}$ the Green's function of $N$. Then, by the choice of $\rho$, for $R$ big 
\begin{equation*}
    \int_{N\setminus B_R(o)} \widehat{G}(o,x) \, d\widehat{\mu}(x) =  \int_{JG^2\setminus B_R(o)} G(o,x) \rho^2(r)\, d\mu(x) <+\infty \,,
\end{equation*}
and by \cite[Corollary 6.7]{Gryg22222} this implies $(ii)$. Consequently, note that also
\begin{align*}
    \int_{0}^\infty \mathcal{N}(p,t) \, dt & = \int_{0}^\infty \int_N H_N(x,p,t) \, d\widehat{\mu}(x) dt  = \int_N \left( \int_0^\infty H_N(x,p,t) \, dt\right)d\widehat{\mu}(x) \\ &= \int_N \widehat{G}(o,x) \, d\widehat{\mu}(x) = \int_{N\setminus B_R(o)} \widehat{G}(o,x) \, d\widehat{\mu}(x) + \int_{B_R(o)} \widehat{G}(o,x) \, d\widehat{\mu}(x) < +\infty \,,
\end{align*}
and since the function $\mathcal{N}(p, \cdot)$ is also nonincreasing this implies that $N$ also satisfies $(iii)$.

    \end{example}

\subsection{Localized asymptotics and proof of Theorem \ref{mainnoncompact}}

  We now show (among other things) that \eqref{eq: theta limit def} is well-posed as in $\R^n$ for manifolds with the $L^\infty-{\rm Liouville}$ property, in the sense that it does not even depend on the choice of $p$. 

  \begin{lemma}\label{lem: ex theta everywhere} Let $(M,g)$ be a complete Riemannian manifold with $\mu(M)=+\infty$ and $E\subset M$ be a set for which the limit \eqref{eq: theta limit def} exists for some $p\in M$. If $M$ has the $L^\infty-{\rm Liouville}$ property, then $\theta_E(p) \equiv \theta_E$ is constant, meaning that the limit in $\theta_E(q)$ exists for all $q\neq p$ and equals $\theta_E(p)$. 
  \end{lemma}
  
  \begin{proof} We adopt the notation in the proof of Theorem \ref{mainrandom0}. In particular, let $ q \mapsto \Theta_{E,s}(q) $ be defined in \eqref{eq: big theta def}. Arguing exactly as in the proof of Theorem \ref{mainrandom0}, every subsequential limit (say, in $C^{2,\alpha}_{\rm loc} (M)$) of $\Theta_{E,s}$ as $s\to 0^+$ is a bounded harmonic function on $M$. 

  Since $M$ has the $L^\infty-{\rm Liouville}$ property, every such subsequential limit is constant. Then, since the limit $\lim_{s\to 0^+} \Theta_{E,s}(p) = \theta_E(p)$ exists by hypothesis, all the subsequential limits must coincide with $\theta_E(p)$ everywhere.
  \end{proof}

  Let us note that the conclusion of Lemma \ref{lem: ex theta everywhere} is not completely trivial in general and is particular of Riemannian manifolds that have the $L^\infty-{\rm Liouville}$ property. Indeed, we believe that on a general complete Riemannian manifold, it can happen that the limit in $\theta_E(\cdot)$ exists for some $p\in M$ but does not exist for some other $q\in M$ with $q\neq p$. See \autoref{sbs: theta existence at diff points} for a brief discussion on this feature.

\begin{lemma}\label{cor50} In the hypothesis of Lemma \ref{lem: ex theta everywhere}, for every bounded $F \subset M$ and $R>0$ with $F\subset B_{R/2}(p)$ there holds
     \begin{equation*}
         \mu(F) \theta_E =  \lim_{s \to 0^+} \J_s(F, E \setminus B_R(p))= \lim_{s \to 0^+} \int_F \int_{E\setminus B_R(p)} \mathcal{K}_s(x,y) \, d\mu(x) d\mu(y) .
     \end{equation*}
\end{lemma}
\begin{proof}
 
    Now since $F \subset B_{R/2}(p)$, we have that $B_{R/10}(y) \subset B_R(p) \subset B_{10R}(y)$ for every $y\in F$. Since the kernel $\mathcal{K}_s$ is nonnegative we get
\begin{equation*}
    \int_{E\setminus B_{10R}(y)} \mathcal{K}_s(x,y) \, d\mu(x) \le \int_{E\setminus B_R(p)} \mathcal{K}_s(x,y) \, d\mu(x) \le  \int_{E\setminus B_{R/10}(y)} \mathcal{K}_s(x,y) \, d\mu(x) \,.
\end{equation*}

By the very definition of $\theta_E$ \eqref{eq: theta limit def} and the fact that the limit does not depend on the radius whenever it exists (see part $(i)$ of Theorem \ref{mainrandom0}) both the left-hand side and right-hand side of the last inequality converge to $\theta_E(y) = \theta_E$, since $\theta_E$ is constant by Lemma \ref{lem: ex theta everywhere}, as $s\to 0^+$. Hence, integrating in $y\in F$ and letting $s\to 0^+$, by dominated convergence 
\begin{equation*}
    \lim_{s \to 0^+} \int_F \int_{E\setminus B_R(p)} \mathcal{K}_s(x,y) \, d\mu(x) d\mu(y) = \int_F \theta_E \, d\mu(y) = \mu(F)\theta_E \,,
\end{equation*}
which is what we wanted to prove.
\end{proof}

\begin{lemma}\label{kkk}
   Let $(M,g)$ be complete with $\mu(M)=+\infty$, and let $A,B \subset M$ be two disjoint measurable sets with $\mu(A), \mu(B)<+\infty$ and with $\J_{s_\circ}(A,B) < +\infty$, for some $s_\circ \in (0,1)$. Then
    \begin{equation*}
        \lim_{s\to 0^+}  \J_s(A,B) =0 \,.
    \end{equation*}
\end{lemma}
\begin{proof}
    First, by Lemma \ref{ex 27} we have
    \begin{equation*}
           \limsup_{s\to 0^+}\J_s(A,B)\leq \limsup_{s\to 0^+} \frac{s}{2} \iint_{A\times B}\int_{1/s}^\infty H_M(x,y,t)\frac{dt}{t^{1+s/2}}d\mu(x)d\mu(y) \,.
    \end{equation*}
    
   Then
    \begin{align*}
         \frac{s}{2} \iint_{A\times B} \int_{1/s}^\infty H_M(x,y,t) & \frac{dt}{t^{1+s/2}}d\mu(x)d\mu(y) = 
         Cs^{1+s/2} \int_{A} \int_{1}^\infty e^{(\xi/s)\Delta}(\chi_B)(x)\frac{d\xi}{\xi^{1+s/2}}d\mu(x) \\ & \le C \int_A \left( s \int_{1}^\infty e^{(\xi/s)\Delta}(\chi_B)(x)\frac{d\xi}{\xi^{1+s/2}} \right) d\mu(x) \,.
    \end{align*}

    Since $\chi_B \in L^1(M)$, for every $x\in A$ (see Remark \ref{heatsolconv}) there holds by dominated convergence
\begin{equation*}
    s \int_{1}^\infty e^{(\xi/s)\Delta}(\chi_B)(x)\frac{d\xi}{\xi^{1+s/2}} \to 0 \,,
\end{equation*}
as $s \to 0^+$. From here, the result follows by dominated convergence using that $\mu(A)<+\infty$.
\end{proof}

The results above directly imply the following. 
\begin{corollary}\label{kkk3} Let $(M,g)$ be complete with $\mu(M)=+\infty$ and with the $L^\infty-{\rm Liouville}$ property, and let  $\Omega \subset M$ be bounded. Then, for every $F\subset \Omega $ with $P_{s_\circ}(F, \Omega) <+\infty$, for some $s_\circ \in (0,1)$, there holds
\begin{equation*}
    \lim_{s \to 0^+} \J_s(F, E \cap \Omega^c) = \mu(F)\theta_E \,.
\end{equation*}
\end{corollary}
\begin{proof}
    Let $p \in M$ and $R \gg 1$ be such that $\Omega \subset B_R(p)$, then
    \begin{align*}
        \J_s(F, E \cap \Omega^c) & =  \J_s(F, E \cap \Omega^c \cap B_R(p)) +  \J_s(F, E \cap \Omega^c \cap B_R^c(p) ) \\ &=  \J_s(F, E \cap \Omega^c \cap B_R(p)) +  \J_s(F, E \cap B_R^c(p) ) \,.
    \end{align*}
    From here, since $\Omega^c \cap B_R(p)$ and $F$ are disjoint and both with finite volume, the first term tends to zero as
     \begin{equation*}
        \J_s(F, E \cap \Omega^c \cap B_R(p)) \le \J_s(F, \Omega^c \cap B_R(p)) \to 0 \,,
    \end{equation*}
    as $s\to 0^+$. Moreover, the second term tends to $\mu(F) \theta_E$ by Lemma \ref{cor50}. 
\end{proof}

 The proof of our main theorem in the infinite volume case is just a simple application of all the results we have derived above. 
 
 \begin{proof}[Proof of Theorem \ref{mainnoncompact}] Write
\begin{align*}
    \frac{1}{2} P_s(E,\Omega) & = \J_s(E\cap \Omega, E^c \cap \Omega ) + \J_s(E\cap \Omega, E^c \cap \Omega^c ) + \J_s(E\cap \Omega^c, E^c \cap \Omega ) \\ &= \frac{1}{2}P_s(E\cap \Omega) - \J_s(E\cap \Omega, E \cap \Omega^c )+\J_s(E^c \cap \Omega, E \cap \Omega^c ) \,.
\end{align*}

By Corollary \ref{sdfgbhsdt} applied to the first term, and by Corollary \ref{kkk3} applied with $F= E\cap \Omega$ and $F=E^c\cap \Omega$ respectively on the second and third term, taking the limit as $s\to 0^+$ we get
     \begin{align*}
         \lim_{s\to 0^+} \frac{1}{2} P_s(E,\Omega) & =  \mu(E\cap \Omega) - \theta_E \mu(E \cap \Omega) +\theta_E \mu(E^c \cap \Omega) \\ &= (1-\theta_E)\mu(E\cap \Omega) +\theta_E\mu(E^c\cap \Omega) \,,
     \end{align*}
     and this shows $(i)$. 
     
     \vsp
     To prove $(ii)$ and $(iii)$ we follow closely the proof of in \cite[Theorem 2.7]{dpfv}, which deals with the analogous property in the case of the Euclidean space $\R^n$. We just sketch the argument since in the reference \cite{dpfv}, the proof is carried on in full detail, and in our case, it is analogous. Let us denote 
     \begin{equation}
         \Theta_{E,s}:= \int_{E \setminus B_R(p)} \mathcal{K}_s(x,p) d\mu(x)   \,,
     \end{equation}
 and fix $R >0 $ such that $\Omega \subset B_{R/2}(p)$. Note that 
 \begin{align*}
    & \int_{\Omega\setminus E} \int_{E\setminus B_R(p)} \mathcal{K}_s(x,y) \, d\mu(x)d\mu(y) - \int_{\Omega\cap E} \int_{E\setminus B_R(p)} \mathcal{K}_s(x,y) \, d\mu(x)d\mu(y) \\[4pt] &=  \frac{1}{2} P_s(E, \Omega) -  \frac{1}{2} P_s(E\cap \Omega, \Omega) - \mathcal{J}_s (\Omega \setminus E, (E\setminus \Omega)\cap B_R(p)) + \mathcal{J}_s(\Omega \cap E, (E\setminus \Omega)\cap B_R(p))   \,.
 \end{align*}
 
 Now, arguing exactly as in the proof of Lemma \ref{cor50} we have that for every $F\subset \Omega $ there holds
 \begin{equation}\label{eq: main noncompact 1}
     \lim_{s\to 0^+} \left| \mu(F)  \Theta_{E,s} - \int_{F} \int_{E\setminus B_R(p)} \mathcal{K}_s(x,y) \, d\mu(x)d\mu(y)  \right| =0 \,.
 \end{equation}
 
 Since $\Omega \setminus E$ and $(E\setminus \Omega)\cap B_R(p)$ are disjoint and both with finite volume (since they are bounded), by Lemma \ref{kkk} we have 
 \begin{equation*}
     \lim_{s\to 0^+} \mathcal{J}_s (\Omega \setminus E, (E\setminus \Omega)\cap B_R(p)) =0 \,,
 \end{equation*}
 and similarly 
 \begin{equation*}
    \lim_{s\to 0^+} \mathcal{J}_s(\Omega \cap E, (E\setminus \Omega)\cap B_R(p)) =0 \,.
 \end{equation*}
 
Hence, taking the limit as $s\to 0^+$ above using \eqref{eq: main noncompact 1} for the left-hand side with $F=\Omega\setminus E$ and $F=\Omega \cap E$ respectively gives
\begin{equation*}
    \lim_{s\to 0^+}  \Theta_{E,s} \big(\mu(\Omega\setminus E) -\mu(\Omega \cap E) \big) = \lim_{s\to 0^+}  \frac{1}{2} \big( P_s(E,\Omega) - P_s(E\cap \Omega, \Omega) \big) \,.
\end{equation*}

Since $E\cap \Omega \subset \Omega$ is bounded, by Corollary \ref{sdfgbhsdt} we have
\begin{equation*}
    \lim_{s\to 0^+}  \frac{1}{2} P_s(E\cap \Omega , \Omega) = \lim_{s\to 0^+}  \frac{1}{2} P_s(E\cap \Omega ) = \mu(E\cap \Omega) \,,
\end{equation*}
 thus 
\begin{equation*}
    \lim_{s\to 0^+}  \Theta_{E,s} \big(\mu(\Omega\setminus E) -\mu(\Omega \cap E) \big) = \left( \lim_{s\to 0^+}  \frac{1}{2}P_s(E,\Omega) \right) - \mu(E\cap \Omega) \,.
\end{equation*}

From here, the conclusion of the theorem easily follows. Indeed, if $\mu(\Omega\setminus E) = \mu(\Omega \cap E)$ then the limit $\lim_{s\to 0^+} \frac{1}{2} P_s(E,\Omega) $ always exists and is equal to $\mu(E\cap \Omega)$. On the other hand, if the limit $\lim_{s\to 0^+}  \frac{1}{2} P_s(E,\Omega) $ exists then from above the limit in $\theta_E$ also exists and there holds
\begin{equation*}
        \theta_E=\frac{\left( \lim_{s\to 0^+} \frac{1}{2} P_s(E, \Omega ) \right)-\mu(E\cap \Omega)}{\mu(\Omega \setminus E)-\mu(E\cap \Omega)} \,,
    \end{equation*}
 and this concludes the proof.
 \end{proof}

%\begin{lemma}
 %   Let $(M,g)$ be a non-compact n-dimensional Riemannian manifold with ${\rm Ric}\geq 0$, then there exists $c(n)>0$ such that
  %  \begin{equation*}
   %     \frac{A(p,R)}{V(p,R)}\geq\frac{c(n)}{R}\quad\forall R>0 \,.
    %\end{equation*}
    %\begin{proof}
     %   We first have
      %  \begin{equation*}
       %     \int_{R}^{2R}A(p,r) dr = V(B_{2R}(p)\setminus B_R(p)).
        %\end{equation*}
        %Then there exists $R_0\in(R,2R)$ such that 
        %\begin{equation*}
         %   A(p,R_0)\geq \frac{ V(p,2R)-V(p,R)}{R}.
      %  \end{equation*}
       % Using Bishop Gromov we know that 
        %\begin{equation*}
         %   A(p,R_0)\leq 2^{n-1}A(p,R),
        %\end{equation*}
        %which together imply
        %\begin{equation}
        %\label{A casina 1}
         %   A(p,R)\geq \frac{1}{2^{n-1}}\frac{V(p,2R)-V(p,R)}{R}=\frac{1}{2^{n-1}}\frac{\mu(B_{2R}(p)\setminus B_R(p))}{R}.
        %\end{equation}
        %$\mu(B_{2R}\setminus B_R)=V(p,2R)-V(p,R)$. Since the manifold is non compact there exists $y\in\partial B_{2R}(p)$ and a geodesic $\gamma$ connecting $p$ and $y$ such that $\gamma(0)=p$ and $\gamma(2R)=y$. Call now $z=\gamma(3R/2)$, then we have ($B_{R/2}(z)\subseteq B_{2R}(p)\setminus B_R(p)$) using the uniform Doubling property of the measure (which is always Bishop-Gromov)
    %\begin{equation}
    %\label{A casina 2}
     %   \mu(B_{2R}(p)\setminus B_R(p))\geq\mu(B_{R/2}(z))\geq \frac{1}{5^n}\mu(B_R(p)).
    %\end{equation}
    %Combining \ref{A casina 1} and \ref{A casina 2} we get
    %\begin{equation*}
     %   A(p,R)\geq\frac{2^{n-1}}{5^n}\frac{\mu(B_R(p))}{R},
    %\end{equation*}
    %concluding the proof with $C=2^{n-1}/5^n$
    %\end{proof}
    
%\end{lemma}

\section{Extension to {\rm RCD} spaces}\label{extension section}

In this section, we briefly explain how our results extend to the case of ${\rm RCD}(K, N)$ spaces, which are a generalization of Riemannian manifolds with an upper bound on the dimension $N$ and Ricci curvature bounded from below by the real number $K$ (and they include weighted manifolds). While assuming the reader familiar with the theory of ${\rm RCD}$ spaces we have to mention at least some references: the introduction of a synthetic lower bound on the Ricci curvature (${\rm CD}$ condition) has been done in the work of Lott and Villani \cite{LottVill09} and in the works of Sturm \cite{KTSI}, \cite{KTSII}. In a subsequent work, Ambrosio, Gigli, and Savarè introduced the ${\rm RCD}$ condition (see \cite{AmGigSav}) to rule out Finsler structures and enforce some Riemannian-like structure at small scales of the space (infinitesimal hilbertianity, see also \cite{Gigli15}).

\vsp 
We stress that we won't reprove every result of the smooth case, but only the ones presenting significant changes needed to perform the asymptotic analysis.

\vsp
First of all, on any ${\rm RCD}(K, N)$ space with $K\in\mathbb{R}$ and $N\in\mathbb{N}\cup\{\infty\}$ it is possible to define a heat kernel and to do so we shall exploit the theory of gradient flows. 

\vsp
We call the \emph{heat flow} $(e^{t\Delta})_{t>0}$ the gradient flow (in the sense of Komura-Brezis theory) of the Cheeger energy, which displays the following properties: for an $L^2$ function $f$ the curve $t\in(0,\infty)\to e^{t\Delta}f\in L^2 $ is locally absolutely continuous, it is such that $e^{t\Delta}f\in{\rm D}(\Delta)$, $\lim_{t\to 0}e^{t\Delta}f=f$ in $L^2$ and satisfies the $\emph{heat equation}$
 \begin{equation*}
     \frac{de^{t\Delta}}{dt}=\Delta e^{t\Delta}f\quad\forall t>0.
 \end{equation*}
 We will now collect some other properties of the heat flow holding on infinitesimally Hilbertian metric measure spaces which we will exploit (see \cite{GiPas20} for a reference):
 \begin{proposition}
 \label{heat flow properties}
     Let $(X,d,\mu)$ be an infinitesimally Hilbertian metric measure space, then we have 
     
     \begin{itemize}
         \item[$(i)$] \emph{(Weak maximum principle)}: Given any $f\in L^2(\mu)$ such that $f\leq C$ $\mu$-almost everywhere we have 
         \begin{equation*}
             e^{t\Delta}f\leq C\quad\mu-{\rm a.e.}
         \end{equation*} 
         \item[$(ii)$] \emph{($e^{t\Delta}$ is self-adjoint):} For all $f,g\in L^2(\mu)$ we have
         \begin{equation*}
             \int_Xe^{t\Delta}f gd\mu=\int_Xe^{t\Delta} g fd\mu\quad\forall t>0.
         \end{equation*} 
        \item[$(iii)$] \emph{($\Delta$ and $e^{t\Delta}$ commute):} For all $f\in{\rm D}(\Delta)$ we have
         \begin{equation*}
             \Delta e^{t\Delta}f=e^{t\Delta}\Delta f\quad\mu-{\rm a.e.},\;\forall t>0.
         \end{equation*} 
     \end{itemize}
    
       Moreover if $(X,d,\mu)$ is an ${\rm RCD}(K,\infty)$ space we have the following additional properties:
     \begin{itemize}
       \item[$(iv)$] \emph{(Bakry-Émery estimate):} For all $f\in W^{1,2}(X)$ and $t>0$ we have
         \begin{equation}
             \label{Bakry-Émery}
             |\nabla e^{t\Delta}f|^2\leq e^{-2Kt}e^{t\Delta}\bigl(|\nabla f|^2\bigr)\quad\mu-{\rm a.e.}
         \end{equation} 
         \item[$(v)$] \emph{($L^\infty-{\rm Lip}$ regularization):} For all $f\in L^\infty(\mu)$ and $t>0$ we have
         \begin{equation}
             \label{Lip regularization}
             \| \nabla (e^{t\Delta}f )\|_{L^\infty(\mu)}\leq\frac{e^{-2Kt}}{\sqrt{t}} \| f \|_{L^\infty(\mu)}.
         \end{equation}
     \end{itemize}
     \end{proposition}

It is then possible to define the heat flow for all probability measures with finite second moment as the $EVI_{K}$ (again, we assume the reader to be familiar with the terminology) gradient flow of the entropy functional. More precisely for every $\mu\in\mathcal{P}_2(X)$, $e^{t\Delta}\mu$ (with a little abuse of notation here) is the unique measure such that 
\begin{equation*}
    \int_X \varphi de^{t\Delta}\mu=\int_Xe^{t\Delta}\varphi d\mu\quad\forall\varphi\in{\rm Lip}_{bs}(X),
\end{equation*}
    where ${\rm Lip}_{bs }(X)$ is the set of Lipschitz functions with bounded support and $e^{t\Delta}\varphi$ is the Lipschitz continuous representative of its equivalence class (which is well-posed thanks to the $L^\infty-{\rm Lip}$ regularization property). 

    On ${\rm RCD}(K,\infty)$ it is possible to define the heat kernel $H_X(x,\cdot,t):=\frac{de^{t\Delta}\delta_x}{d\mu}$ and we have the following (see \cite{JLZ16} for a reference):

    \begin{proposition}
        Let $(X,d,\mu)$ be an ${\rm RCD}(K,N)$ space with $N\in\mathbb{N}$, then for all $\epsilon >0$, for some $C_1,C_2,C_3,C_4$ nonnegative constants (possibly depending on $\epsilon$ and $N$) we have
        \begin{equation}
            \label{Heat kernel bounds}
            \frac{1}{C_1\mu(B_{\sqrt{t}}(y))}\exp{\biggl(-\frac{d^2(x,y)}{(4-\epsilon)t}-C_2t\biggr)} \leq H_X(x,y,t)\leq \frac{C_1}{\mu(B_{\sqrt{t}}(y))}\exp{\biggl(-\frac{d^2(x,y)}{(4+\epsilon)t}+C_2t\biggr)} 
        \end{equation}
        for all $x,y\in X$, $t>0$ and
        \begin{equation}
            \label{Gradient heat kernel bound}
            |\nabla H_X(x,\cdot,t)|(y)\leq\frac{C_3}{\sqrt{t}\mu(B_{\sqrt{t}}(y))}\exp{\biggl(\frac{-d^2(x,y)}{(4+\epsilon)t}-C_4t\biggr)}
        \end{equation}
        $\mu\times\mu$-a.e. $(x,y)\in X\times X$, for all $t>0$.

        \vsp 
     Moreover, if $K=0$ then estimate \eqref{Heat kernel bounds} holds with $C_2=C_4 = 0 $.
    \end{proposition}
On any ${\rm RCD}(K,\infty)$ space we have
\begin{equation*}
    \int_X H_X(x,y,t)d\mu(x)=1
\end{equation*}
for all $y\in X$, $t>0$. That is, $X$ is stochastically complete.

\vsp
    In the setting of ${\rm RCD}(K,N)$ (actually infinitesimal hilbertianity is not required) we also have Bishop-Gromov's comparison theorem, holding both for the perimeter measure and the volume measure (see \cite{KTSII}). Finally, it is possible to prove that the following version of the Harnack inequality holds (see \cite{FreeHarnack16} for the proof) 
    \begin{proposition}[Harnack inequality]\label{Harnack inequality}
    Let $(X,d,\mu)$ be an ${\rm RCD}(K,\infty)$ space, $p\in(1,\infty)$ and $f\in L^1(\mu)+ L^\infty(\mu)$, then 
    \begin{equation*} 
        |(e^{t\Delta}f)(x)|^p\leq (e^{t\Delta}|f|^p)(y)\exp\biggl(\frac{pKd^2(x,y)}{2(p-1)(e^{2Kt}-1)}\biggr)
    \end{equation*}
    for all $x,y\in X\times X$ and $t>0$.
        
    \end{proposition}
    From the previous Harnack inequality, it is possible to prove the following Gaussian bound (see \cite[Theorem 4.1]{tamanini2019harnack}) for ${\rm RCD}(K,\infty)$ spaces (compare with \eqref{Heat kernel bounds} above for ${\rm RCD}(K,N)$ spaces).
    
    \begin{proposition}
        Let $(X,d,\mu)$ be an ${\rm RCD}(K,\infty)$ space, then there exists $C_K>0$ and for all $\varepsilon>0$ there exists $C_\varepsilon>0$ such that 
        \begin{equation}
            \label{Gaussian upper bound infinite}
            H_{X}(x,y,t)\leq \frac{1}{\sqrt{\mu(B_{\sqrt{t}}(x))}\sqrt{\mu(B_{\sqrt{t}}(y))}}\exp\biggl(C_\varepsilon(1+C_K)t-\frac{d^2(x,y)}{(4+\varepsilon)t}\biggr).
        \end{equation}
        If $K\geq 0$ one can take $C_K=0$.
    \end{proposition}
    The second ingredient we need is a generalization to ${\rm RCD}(K,\infty)$ spaces of the $L^2-{\rm Liouville}$ property of Yau (our Theorem \ref{dfsdfs}).
    \begin{proposition}
    \label{Generalized Yau}
        Let $(X,d,\mu)$ be an ${\rm RCD}(K,\infty)$ space. Then, any $L^2(\mu)$ harmonic function is constant.
    \end{proposition}
    \begin{proof}
       Denote $w(t,x):=e^{t\Delta}u (x)$. Assume $u\in L^2(\mu)$ is harmonic, then by applying the heat flow to $\Delta u=0$ and using item $(iii)$ of Proposition \ref{heat flow properties} we have 
        \begin{equation*}
            \Delta w =0.
        \end{equation*}
       By gradient flow theory we have
        \begin{equation*}
            \int_{X}|\nabla w|^2d\mu\leq\frac{1}{2t}\int_X|u|^2d\mu,
        \end{equation*}
        whence
        \begin{equation*}
            0=-\int_X w \Delta w \, d\mu = \int_X |\nabla w |^2d\mu.
        \end{equation*}
        This means $|\nabla w|=0$ $\mu$-a.e. and by the Sobolev to Lipschitz property, this implies that $w$ is constant. Therefore there exists $C=C(t)$ such that $w(t, \cdot )=C(t)$. 
        
        Now if $\mu(X)<+\infty$ we can infer (as $u\in L^2(\mu)$ implies $u\in L^1(\mu)$)
        \begin{equation*}
            \int_{X} w\, d\mu = \int_{X} u \, d\mu = \mu(X)C(t),
        \end{equation*}
        hence $C$ does not actually depend on $t$ and by taking the limit as $t\to 0^+$ we infer that $u$ is constant. 
        
        If $\mu(X)=+\infty$, then for every $t$, we have $w=0$ because the only constant in $L^2(\mu)$ is zero, and we conclude.
    \end{proof}
    \begin{remark}
        The previous proposition actually does not require a curvature condition: working in a space in which having zero weak upper gradient implies being constant suffices.
    \end{remark}
We then have the following result, which is a non-smooth analog of Proposition \ref{gryg}.
\begin{proposition}
    Let $(X,d,\mu)$ be an ${\rm RCD}(K,\infty)$ space, then we have the following dicotomy:
    \begin{itemize}
        \item[$(i)$] If $\mu(X)< + \infty$ then 
        \begin{equation*}
            H_X(t,x,y)\to\frac{1}{\mu(X)}\quad{\rm as}\;t\to\infty\;\;\forall x,y\in X.
        \end{equation*} 
        \item[$(ii)$] If $\mu(X)= + \infty$ we have
        \begin{equation}
        \label{unifconvatinf}
            H_{X}(\cdot,\cdot,t)\to 0\quad{\rm as}\;t\to\infty
        \end{equation}
         locally uniformly and $H_X(p,\cdot,t)\to 0$ uniformly as $t\to\infty$ for every $p\in M$.
    \end{itemize}
\end{proposition}
\begin{proof}
   The proof follows along the same lines of Proposition \ref{gryg}. If $\mu(X)<+\infty$ let $f=H_{X}(p,\cdot,1)-1/\mu(X)$, otherwise let $f=H_{X}(p,\cdot,1)$, then 
    $\max\{\|e^{t\Delta}f\|_{L^1},\|e^{t\Delta}f\|_{L^\infty}\}\leq C$ due to the properties of the heat flow. Moreover by the semigroup property of it is easy to see that weak convergence in $L^2(\mu)$ of $e^{t\Delta}f$ is equivalent to strong convergence and we again have the inequality
    \begin{equation*}
        |(e^{t\Delta}f,g)|\leq |(e^{t\Delta}f,f)\|(e^{t\Delta}g.g)|\leq \|g\|^2_{L^2(\mu)}|(e^{t\Delta}f,f)|
    \end{equation*}
    for all $t\in (0,\infty)$ and for all $f,g\in L^2(\mu)$. Using the spectral measure representation and Proposition \ref{Generalized Yau}, we infer the desired $L^2$ convergence. This convergence can be upgraded to be locally uniform by the Harnack inequality (Proposition \ref{Harnack inequality}) with $p=2$ and by the fact that $|f|^2\leq \|f\|_{L^\infty}|f|$, together with the maximum principle to get
    \begin{equation*}
       |e^{t\Delta} f(x)|^2\leq \|f\|_{L^\infty}e^{t\Delta}(|f|)(y)\exp\biggl(\frac{2KR^2}{2(e^{2Kt}-1)}\biggr)
    \end{equation*}
for every $y\in B_R(x)$. Integrating over the latter set in $d\mu(y)$ and taking the supremum allows to conclude. The global uniform convergence follows as in the smooth case.
\end{proof}
\begin{remark}
    As in the smooth case if $\mu(X) = +\infty$ we have that for every $f\in L^1(\mu)$
    \begin{equation*}
        \lim_{t\to\infty}e^{t\Delta}f(x)=0
    \end{equation*}
    for every $x\in X$.
\end{remark}

We refer to \cite{Sire} for an introduction to $H^s$ spaces on very general ambient space, like ${\rm RCD}$ spaces and more. We have the analog of Theorem \ref{globlimcomp}. 
\begin{theorem}
    Let $(X,d,\mu)$ be an ${\rm RCD}(K,\infty)$ space with $K\in\mathbb{R}$ and $\mu(X)< + \infty$. Let $u\in H^{s_\circ/2}(X)$ for some $s_\circ \in (0,1)$ with bounded support. Then 
    \begin{equation*}
        \lim_{s\to 0^+}\frac{1}{2}[u]_{H^{s/2}(X)}^2=\|u\|_{L^2(X)}^2-\frac{1}{\mu(M)}\biggl(\int_Xu d\mu\biggr)^2.
    \end{equation*}
\end{theorem}
\begin{proof}
    The proof is exactly the same as in the smooth case exploiting the $L^2-{\rm Liouville}$ property of Proposition \ref{Generalized Yau}.
\end{proof}
To prove the convergence result for the case of infinite volume we need a convergence result for the solution of the heat equation to the initial datum.
We, therefore, recall the following (upper) Large Deviation Principle on proper ${\rm RCD}(K,\infty)$ spaces (see \cite[Theorem 5.3]{gigli2022viscosity})
\begin{theorem}
    Let $(X,d,\mu)$ be a proper ${\rm RCD}(K,\infty)$ space, then for every $x\in X$ and closed set $C\subseteq X$ we have, setting $\mu_t[x]=H_X(\cdot,x,t)\mu$,
    \begin{equation}
        \label{Large Deviation}
        \limsup_{t\to 0}t\log(\mu_t[x](C))\leq -\inf_{y\in C}\frac{d^2(x,y)}{4}.
    \end{equation}
\end{theorem}
\begin{remark}
    In \eqref{Large Deviation} we can choose $C=X\setminus B_r(p)$ and obtain the following estimate for small times (depending on $r>0$ and $\varepsilon>0$)
    \begin{equation}
    \label{expconv}
        |e^{t\Delta}(\chi_{X\setminus B_r(p)})(p)|\leq \exp\bigg(-\frac{r^2-\varepsilon}{4t}\bigg)
    \end{equation}
\end{remark}

We are finally ready to prove the following proposition (analog of Proposition \ref{limexistence})
\begin{proposition}
    Let $(X,d,\mu)$ be a proper ${\rm RCD}(K,\infty)$ space with $\mu(X)= + \infty$. Then for every $p\in X$
    \begin{equation*}
        \theta_M(p)=\lim_{s\to 0}\int_{X\setminus B_1(p)}\mathcal{K}_s(x,p)d\mu(x) = 1.
    \end{equation*}
\end{proposition}
\begin{proof}
    As for the smooth case, we first show that 
    \begin{equation*}
        \lim_{s\to 0^+}\frac{s}{2}\int_{X\setminus B_1(p)}\int_0^1 H_{X}(x,p,t)\frac{dt}{t^{1+s/2}} d\mu(x)=0.
    \end{equation*} 
    Indeed there exists $\delta>0$ such that for all $t\leq\delta$ \eqref{expconv} holds, so that the previous integral can be estimated with the following
    \begin{equation*}
        \frac{s}{2}\int_{0}^\delta e^{-r^2/5t}\frac{dt}{t^{1+s/2}}+ \frac{s}{2}\int_{X\setminus B_1(p)}\int_\delta^1H_X(x,p,t)\frac{dt}{t^{1+s/2}}.
    \end{equation*}
    The first term clearly goes to zero as $s\to 0^+$ and to handle the second we use Fubini to deduce that (here stochastical completeness is not necessary but ${\rm RCD}(K,\infty)$ spaces enjoy this property so we write the equality sign)
    \begin{equation*}
        \frac{s}{2}\int_{X\setminus B_1(p)}\int_\delta^1H_X(x,p,t)\frac{dt}{t^{1+s/2}}=\frac{s}{2}\int_\delta^1\frac{dt}{t^{1+s/2}}-\frac{s}{2}\int_{B_1(p)}\int_\delta^1H_X(x,p,t)\frac{dt}{t^{1+s/2}}d\mu(x).
    \end{equation*}
    Again the first term trivially goes to zero while for the second we apply \eqref{Gaussian upper bound infinite} and exploit properness of the space to infer that $H_X(\cdot,\cdot,\cdot)$ is equibounded in $B_1(p)\times[\delta,1]$ so that  
    \begin{equation*}
        \limsup_{s\to 0}\biggl|\frac{s}{2}\int_{B_1(p)}\int_\delta^1 H_X(x,p,t)\frac{dt}{t^{1+s/2}}d\mu(x)\biggr|\leq\limsup_{s\to 0^+} C\frac{s}{2}\int_\delta^1\frac{dt}{t^{1+s/2}}=0\,.
    \end{equation*}
    We now claim that 
    \begin{equation*}
        \lim_{s\to 0^+}\frac{s}{2}\int_{B_1(p)}\int_1^\infty H_X(x,p,t)\frac{dt}{t^{1+s/2}}d\mu(x)=0.
    \end{equation*}
    Indeed, thanks to the local uniform convergence proved in \eqref{unifconvatinf} and reasoning as in the previous step the latter result easily follows. 

    \vsp
    Finally, we can perform the same steps and write
    \begin{equation*}
         \theta_M(p)=\lim_{s\to 0}\frac{s}{2}\int_{X}\int_1^\infty H_X(x,y,t)\frac{dt}{t^{1+s/2}}d\mu(x),
    \end{equation*}
    which equals $1$ by using stochastical completeness.
\end{proof}
In the following proposition, we study the behavior of the singular kernel $\mathcal{K}_s(x,y)$.
\begin{proposition}\label{Singular kernel bounds RCD}
    Let $(X,d,\mu)$ be an ${\rm RCD}(K,N)$ space with $\mu(X)=+\infty$ and essential dimension equal to $n$. Then, for every $x\in X$ which is a regular point we have 
    \begin{equation}
       \frac{C s}{r^{n+s}}\leq \mathcal{K}_s(x,y)\leq \frac{C s}{r^{n+s}}+o_s(1)+\sup_{t\geq 1/s}H_M(x,y,t),
    \end{equation}
   for every $y\in X$, where $r=d(x,y)$. In particular $\mathcal{K}_s(x,\cdot)\to 0$ as $s\to 0^+$ locally uniformly away from $x$.
\end{proposition} 
\label{Singular kernel RCD}
\begin{proof}
    Let us define
    \begin{align*}
        \mathcal{K}_s(x,y) & =\frac{s}{2}\int_0^1H_X(x,y,t)\frac{dt}{t^{1+s/2}}+\frac{s}{2}\int_1^{1/s} H_X(x,y,t)\frac{dt}{t^{1+s/2}}+\frac{s}{2}\int_{1/s}^\infty H_X(x,y,t)\frac{dt}{t^{1+s/2}} \\ & =:I_1+I_2+I_3.
    \end{align*}
    By the Gaussian estimates \eqref{Heat kernel bounds} and using the fact that $x$ is a regular point we have
    \begin{equation*}
    I_1\leq C s \int_0^1 e^{-r^2/5t}\frac{dt}{t^{1+s/2+n/2}}\leq \frac{C s}{r^{n+s}}.
\end{equation*}
Moreover, since $\mu(X)=+\infty$ by \eqref{unifconvatinf} the heat kernel converges locally uniformly to zero, and we also get
\begin{equation*}
    I_2\leq Cs \int_1^{1/s}\frac{dt}{t^{1+s/2}}= C(1-s^{s/2}) = o_s(1),
\end{equation*}
for some constant $C$ which is bounded in a neighborhood of $x$.
Finally, we have
\begin{equation*}
    I_3\leq\frac{s}{2}\sup_{t\geq 1/s}H_M(x,y,t)\int_{1/s}^\infty\frac{dt}{t^{1+s/2}},
\end{equation*}
thus proving the upper bound in \eqref{Singular kernel bounds RCD}. For the lower bound it is enough to neglect $I_2$ and $I_3$ and apply the Gaussian estimate from below to $I_1$. 

\vsp
Finally, the local uniform convergence $\mathcal{K}_s(x,\cdot)\to 0$ is apparent due to the local uniform convergence \eqref{unifconvatinf} of the heat kernel to zero and the other quantities involved. 
\end{proof}

With the next proposition, we show that the heat density of a set, whenever it exists, is independent of the radius and also on the point if the $L^\infty-{\rm Liouville}$ property holds, analogously to the case of manifolds. 
\begin{proposition}
    Let $(X,d,\mu)$ be an ${\rm RCD}(K,\infty)$ space with $\mu(X)= + \infty$, let $E\subset M$ be measurable and set
    \begin{equation*}
        \Theta_{E,s}(p,r) :=\int_{X\setminus B_r(p)}  \mathcal{K}_s (p,x)d\mu(x).
    \end{equation*}
    Then for all $0<r\leq R$ one has
    \begin{equation*}
\limsup_{s\to 0^+} \big|  \Theta_{E,s}(p,R) -  \Theta_{E,s}(p,r) \big|=0.
    \end{equation*}
    meaning that if $\lim_{s\to 0^+}  \Theta_{E,s}(p,r)=\theta_E(p)$ exists for some $p\in M$, then it does not depend on $r$. Moreover, if the $L^\infty-{\rm Liouville}$ property holds on $X$ and $\theta_E(p)$ exists for some $p\in X$, then $\theta_E\equiv\theta_E(p)$ is constant.
\end{proposition}
\begin{proof}
    We first show the independence on the radius; therefore we fix any two $0<r<R$, and we show that 
    \begin{equation*}
        \limsup_{s\to 0^+}\frac{s}{2}\int_{\overline{B_{R}(p)}\setminus B_r(p)}\int_0^\infty H_X(x,p,t)\frac{dt}{t^{1+s/2}}d\mu(x)=0.
    \end{equation*}
    We split the integral over time in three pieces: one from $0$ to $\varepsilon$, one from $\varepsilon$ to $T$, and the last one from $T$ to $\infty$. The first piece goes to zero since $\overline{B_R(p)}\setminus B_r(p)$ is a closed set and we can apply \eqref{expconv}, the second piece goes to zero for every $T\gg 1$ thanks to the properness of the space, the Gaussian upper bound \eqref{Gaussian upper bound infinite} and easy calculations, while the last piece is such that, for all $T\geq T_0(\varepsilon)$
    \begin{equation*}
        \limsup_{s\to 0^+}\frac{s}{2}\int_T^\infty H_X(x,p,t)\frac{dt}{t^{1+s/2}}d\mu(x)\leq\varepsilon.
    \end{equation*}
    Since this holds for every $\varepsilon$ we get the convergence to zero.

\vsp 
    For what concerns the independence on the point, we first take $r$ big that $q\in B_{r/10}(p)$ and wlog, we assume $E$ to be closed. We have
    \begin{align*}
        \limsup_{s\to 0^+}\biggl|\int_{E\setminus B_r(p)}  \mathcal{K}_s(x,p) & d\mu(x)-\int_{E\setminus B_{2r}(q)} \mathcal{K}_s(x,q)d\mu(x)\biggr| \\ & \leq \limsup_{s\to 0^+} \biggl|\int_{E\setminus B_r (p)}\mathcal{K}_s(x,q)d\mu(x)-\int_{E\setminus B_{2r} (q)}\mathcal{K}_s(x,q)d\mu(x)\biggr| \\ &
       + \limsup_{s\to 0^+}\biggl|\int_{E\setminus B_r(p)}\mathcal{K}_s(x,p)-\mathcal{K}_s(x,q)d\mu(x)\biggr|=: I_1+I_2.
    \end{align*}
    The first integral is zero since
    \begin{equation*}
        I_1 \le \limsup_{s\to 0^+} \int_{B_{2r}(q) \setminus B_{r}(p)} \mathcal{K}_s(x,q)d\mu(x) \le \theta_{B_{2r}(q)}(q) =0 \,,
    \end{equation*}
    where we have used the independence on the radius. 
    
    While for $I_2$ we shall exploit the $L^\infty-{\rm Liouville}$ property of $X$. We can, as usual, expand the singular kernel and split the integral in time into three pieces in time, one going from $0$ to $1$, another from $1$ to $T\gg 1$, and lastly, from $T$ to $\infty$. The first two are handled thanks to the exponential convergence \eqref{expconv} and the boundedness of the heat kernel, while for the last one, we have
    \begin{equation*}
        \limsup_{s\to 0^+}\bigg|\int_T^\infty e^{t\Delta}(\chi_{E\setminus B_r(p)})(p)-e^{t\Delta}(\chi_{E\setminus B_r(p)})(q)\frac{dt}{t^{1+s/2}}\bigg|=0 \,,
    \end{equation*}
    thanks to the $L^\infty-{\rm Liouville}$ property. 
    
    Indeed $e^{t\Delta}(\chi_{E\setminus B_r(p)})$ converges up to subsequences to a constant harmonic function; hence its (of the limit function) value at the points $p$ and $q$ is the same so that, being this true for any subsequence, $e^{t\Delta}(\chi_{E\setminus B_r(p)})(p)-e^{t\Delta}(\chi_{E\setminus B_r(p)})(q)\to 0$ as $t\to\infty$.
\end{proof}
\begin{remark}
    In the previous proposition, we only care about spaces satisfying the $L^\infty-{\rm Liouville}$ property. However, with a little work, it is possible to show that the function $p\mapsto\theta_E(p)$, whenever it exists, is a bounded harmonic function in a suitable weak sense. 
\end{remark}
Finally, we have the analog of Theorem \ref{globlimnoncompact}.
\begin{theorem}\label{rcdlimit22}
    Let $(X,d,\mu)$ be a proper ${\rm RCD}(K,N)$ space with $\mu(X)=+\infty$, $N<+\infty$, essential dimension equal to $n$ and let $s_\circ\in (0,1)$. Then for every $u \in H^{s_\circ/2}(X)\cap L^\infty(X)$ with bounded support there holds
    \begin{equation*}
         \lim_{s \to 0^+} \frac{1}{2} [u]_{H^{s/2}(X)}^2 =  \|u \|_{L^2(X)}^2 .
     \end{equation*}
\end{theorem}
 \begin{proof}
     The proof is similar to the smooth case; we just need to handle the computations more carefully. We advise the reader to first see the proof in the smooth case of Theorem \ref{globlimnoncompact}. 

\vsp
     By Proposition \ref{coincideprop2} (which also holds for ${\rm RCD}$ spaces, see Remark \ref{rcdrmk}) for $\mu$-a.e. $x\in X$ the integral in $(-\Delta)^{s/2}_{\rm Si} u$ is absolutely convergent. Fix $x\in X$ in this full-measure set and $R>0$ such that $ \supp(u) \subseteq B_R(x)$. Now we prove that, as $s\to 0^+$, $(-\Delta)^{s/2}_{\rm Si} u\to u$  $\mu$-a.e. with the same strategy of the smooth case. Take also $x\in X$ to be a regular point, we have 
     \begin{equation*}
         (-\Delta)^{s/2}_{\rm Si} u(x) = \int_{B_R(x)}(u(x)-u(y))\mathcal{K}_s(x,y)d\mu(y)+u(x)\int_{X\setminus B_R(x)}\mathcal{K}_s(x,y)d\mu(y)
     \end{equation*}
     and we are left to prove that the first term goes to zero as $s\to 0^+$, as the second one in the limit is precisely $u(x)$. Now fix $\rho \ll 1 $ and let us split the first integral as follows
     \begin{align*}
   \bigg| \int_{B_R(x)}  (u(x)-u(y)) \mathcal{K}_s(x,y) \, d\mu(y) \bigg|   &=   \int_{ B_{\rho}(x)} |u(x)-u(y)|\mathcal{K}_s(x,y) \, d\mu(y) \\ & \s\s + \int_{B_R(x) \setminus B_{\rho}(x) } |u(x)-u(y)|\mathcal{K}_s(x,y) \, d\mu(y).
   \end{align*}
     For the first integral, we can  apply Proposition \ref{Singular kernel bounds RCD} to obtain
     
     \begin{equation}
     \label{Going to zero stuff}
           \int_{ B_{\rho}(x)} |u(x)-u(y)|\mathcal{K}_s(x,y) \, d\mu(y)\leq C s \int_{B_\rho(x)}\frac{|u(x)-u(y)|}{d(x,y)^{n+s}}d\mu(y)+o_s(1).
     \end{equation}
   Applying H\"older inequality as in the smooth case (take $s$ small so that $2s<s_0$) we now get
    \begin{equation*}
     \int_{ B_{\rho}(x)} \frac{|u(x)-u(y)|}{d(x,y)^{n+s}} \, d\mu(y)\leq  \left( \int_{B_\rho(x)} \frac{(u(x)-u(y))^2}{d(x,y)^{n+s_\circ} } \, d\mu(y) \right)^{1/2} \left( \int_{B_\rho(x)} \frac{1}{d(x,y)^{n+2s-s_\circ} } \,d\mu(y)  \right)^{1/2}
    \end{equation*}
    and conclude in the same way that taking the limit as $s\to 0^+$ in \eqref{Going to zero stuff} gives zero. For the second term, we just use the fact that $\mathcal{K}_s(x,\cdot)$ goes to zero locally uniformly away from $x$ together with dominated convergence. Therefore we have proved that $(-\Delta)^{s/2}_{\rm Si} u\to u$ $\mu$-a.e. as $s\to 0^+$. To establish the seminorms' convergence, we exploit Corollary \ref{corolequivalence}, which also holds in this non-smooth setting with the same proof. To conclude we just need to prove that $(-\Delta)^{s/2}_{\rm Si} u \rightharpoonup u$ weakly in $L^2(\mu)$: this is however apparent because of the equiboundedness of $\|(-\Delta)^{s/2}_{\rm Si} u\|_{L^2(\mu)}$ given by the estimate \eqref{uqweriufgweur}.
 \end{proof}
Thanks to the previous results we would be in the position of stating and proving (which we won't do since the proofs are exactly the same as in the smooth case) the theorems regarding the asymptotics of the fractional perimeter Theorem \ref{maincompact} and Theorem \ref{mainnoncompact}, also in this non-smooth setting.  
\section{Appendix}

\subsection{On the existence/nonexistence of $\theta_E(\cdot)$ at different points.}\label{sbs: theta existence at diff points}

Let $(M,g)$ be a complete Riemannian manifold with infinite volume and $E\subset M$. As we have proved in Lemma \ref{lem: ex theta everywhere} if $M$ has the $L^\infty-{\rm Liouville}$ property and $\theta_E(p)$ exists for some $p\in M$ then it exists for all $p\in M$ and the two values coincide. Let us stress that, on manifolds with $L^\infty-{\rm Liouville}$ property, the limit does not need to exist, but if it does not exist at some point, then it does not exist everywhere. For example, even on $\R^n$ in \cite[Example 2.8]{dpfv}, the authors exhibit a set for which the limit $\theta_E(x)$ does not exist at every point $x\in \R^n$. 

\vsp
On the other hand, on a general $M$ without the $L^\infty-{\rm Liouville}$ property, we believe that $\theta_E(\cdot)$ could exist for some $p\in M$ and fail to exist for some $q\neq p$. Let $\Theta_{E,s}(\cdot)$ defined as in \eqref{eq: big theta def} so that $\theta_E(p) = \lim_{s\to 0^+} \Theta_{E,s}$, if the limit exists. It can be proved with the Li-Yau Harnack inequality \cite[Corollary 12.3]{Liga} that if the limit in $\theta_E(p)$ does not exist and 
\begin{equation*}
    \limsup_{s\to 0^+} \Theta_{E,s}(p) - \liminf_{s\to 0^+} \Theta_{E,s}(p) = \delta>0 \,,
\end{equation*}
then the limit still does not exist for every $q\in B_{C\delta}(p)$, where $C>0$ is a constant that depends on $M$. But in this estimate the lower bound for $\limsup_{s\to 0^+} \Theta_{E,s}(q) - \liminf_{s\to 0^+} \Theta_{E,s}(q)$ tends to $0$ as $q$ approaches $ \partial B_{C\delta}(p)$. Without further information in $M$, we do not see any reason why the limit should not exist at some point outside $B_{C\delta}(p)$. 

\begin{comment}
\begin{lemma}
    Let $(X,d,\mu)$ be an ${\rm RCD}(K,\infty)$ space and $E\subset X$ be a Borel set. Assume that $p\mapsto\theta_E(p)$ exists (meaning that the limit in \eqref{eq: theta limit def} exists) for $\mu$-a.e. $p\in X$, then $\theta_E$ is an harmonic function.
\end{lemma}

\begin{proof}
    First of all it is easy to see that $0\leq\theta_E(p)\leq 1$, whence $\theta_E\in L^\infty(\mu)$. Secondly we see that for all $\varphi\in {\rm \widetilde{Test}}(X)$ we have 
    \begin{equation*}
        \int_X\theta_E\Delta\varphi \, d\mu = 0 \,,
    \end{equation*}
    where ${\rm \widetilde{Test}}(X)$ is defined as in \cite[page 17]{Gig23}.
    Now we claim that $e^{t\Delta}(\theta_E) =\theta_E$. Indeed, we observe that here point (i) of \cite[Lemma 4.2]{Gig23} applies with $g=0$. Therefore using the latter both for $\theta_E$ and for $-\theta_E$ allows to conclude that $e^{t\Delta}\theta_E=\theta_E$. In view of this we get that $\theta_E$ is Lipschitz continuous thanks to \eqref{Lip regularization} and therefore
    \begin{equation}
    \label{harmonic in the sense of distributions}
        \int_{X}\nabla\theta_E\cdot\nabla\varphi \,  d\mu=0 \,,
    \end{equation}
    for all $\varphi\in{\rm \widetilde{Test}}(X)$. Finally we can exploit the density of ${\rm \widetilde{Test}}(X)$ in ${\rm Lip}_{bs}(X)$ to conclude that \eqref{harmonic in the sense of distributions} holds for any $\varphi\in {\rm Lip}_{bs}(X)$. This proves that $\theta_E$ is harmonic and concludes the proof.
\end{proof}
\end{comment}
\subsection{Heat kernel estimates and \texorpdfstring{$H^s(M)$}{Hs} spaces.}\label{hker est sbs}

Here $(M,g)$ denotes a complete, connected Riemannian manifold. First, we present a simple interpolation inequality for $H^{s/2}(M)$ spaces.

\vsp
This inequality is known in the case of $M=\R^n$ or $M=\Omega \subset \R^n$ for fractional Sobolev spaces $W^{s,p}$, also when $p\neq 2$. Here, we carry on a structural proof using a few properties of the heat kernel, which gives the interpolation inequality on general ambient spaces.   
\begin{lemma}\label{intineq}
    Let $u\in H^{\sigma}(M)$ for some $\sigma\in(0,1)$, and let $0<s<\sigma<1$. Then $u\in H^s(M)$ and the following inequality holds
    \begin{equation*}
        [u]_{H^{s}(M)}\leq C\|u\|_{L^2(M)}^{1-s/\sigma  }[u]_{H^{\sigma}(M)}^{s/\sigma}.
    \end{equation*}
    for some absolute constant $C>0$.
    \begin{proof}
        We have 
        \begin{align*}
             |\Gamma(-s)|[u]_{H^{s}(M)}^2 &=  \iint_{M\times M} (u(x)-u(y))^2\int_0^\infty H_M(x,y,t)\frac{dt}{t^{1+s}}d\mu(x)d\mu(y)  \\
             &\leq \iint_{M\times M} (u(x)-u(y))^2\int_0^\xi H_M(x,y,t)\frac{dt}{t^{1+s}}d\mu(x)d\mu(y) \\
             &+\iint_{M\times M} (u(x)-u(y))^2\int_\xi^\infty H_M(x,y,t)\frac{dt}{t^{1+s}}d\mu(x)d\mu(y) 
        \end{align*}
        where $\xi\in (0,\infty)$ will be chosen at the end. Note that for all $t\in(0,\xi)$ we have $(\xi/t)^{1+s}\leq(\xi/t)^{1+\sigma}$ so that we can estimate from above the first integral of the previous inequality with
        \begin{equation*}
             \xi^{\sigma-s}\iint_{M\times M} (u(x)-u(y))^2\int_0^\xi H_M(x,y,t)\frac{dt}{t^{1+\sigma}}d\mu(x)d\mu(y)\leq  \xi^{\sigma-s} |\Gamma(-\sigma)|[u]^2_{H^{\sigma}(M)}.
        \end{equation*}
        The symmetry of the heat kernel and the fact that $\mathcal{M}(t,y)\leq 1$, for all $y\in M$, together imply that the second integral can be bounded by
        \begin{equation*}
           \iint_{M\times M} (u(x)-u(y))^2\int_\xi^\infty H_M(x,y,t)\frac{dt}{t^{1+s}}d\mu(x)d\mu(y) \le \frac{4}{s\xi^{s}}\|u\|^2_{L^2(M)}.
        \end{equation*}
        These two inequalities lead to
        \begin{equation*}
            |\Gamma(-s)|[u]_{H^{s}(M)}^2\leq  \xi^{\sigma-s} |\Gamma(-\sigma)| [u]^2_{H^{\sigma}(M)} + \frac{4}{s\xi^{s}}\|u\|^2_{L^2(M)}.
        \end{equation*}
    Optimizing the right-hand side in $\xi$ gives that the optimal value is
    \begin{equation*}
        \xi=\biggl(\frac{4\|u\|_{L^2(M)}^2}{(\sigma-s)|\Gamma(-\sigma)|[u]^2_{H^{\sigma}(M)}}\biggr)^{1/\sigma}.
    \end{equation*}
    Putting everything together gives
    \begin{equation*}
        |\Gamma(-s)|[u]^2_{H^{s}(M)}\leq \frac{C}{s} \|u\|_{L^2(M)}^{2(1-s/\sigma)}[u]_{H^{\sigma}(M)}^{2s/\sigma} \,, 
    \end{equation*}
   and this implies
   \begin{equation*}
        [u]_{H^{s}(M)}\leq C \|u\|_{L^2(M)}^{1-s/\sigma}[u]_{H^{\sigma}(M)}^{s/\sigma} \,, 
    \end{equation*}
    as desired.
    \end{proof}
\end{lemma}

\begin{lemma}[\hspace{1sp}\cite{CFS23}]\label{convergence} Let $(M^n,g)$ be a complete $n$-dimensional Riemannian manifold and let $ B_R(p) \subset M $. Then 
\begin{equation*}
   e^{t\Delta}(\chi_{M\setminus B_R(p)})(p) = \int_{M\setminus B_R(p)} H_M(x,p,t) \, d\mu(x) \le C e^{-c/t} \,,
\end{equation*}
for some $C,c>0$ depending on $R$ and the geometry of $M$ in $B_R(p)$. 
\end{lemma}
\begin{proof}
    This is essentially \cite[Lemma 2.9]{CFS23}. Indeed, in \cite[Lemma 2.9]{CFS23} the authors prove that if $(M,g)$ is a complete Riemannian manifold and $B_r(p) \subset M$ is a ball diffeomorphic to $\B_r(0) \subset T_pM $ with metric coefficients $g_{ij}$ (say, in normal coordinates) uniformly close to $\delta_{ij}$, then 
    \begin{equation*}
        \int_{M\setminus B_r(p)} H_M(x,p,t) \, d\mu(x) \le C e^{-c\,r^2/t} \,,
    \end{equation*}
    for some $C,c>0$ dimensional. Then, taking $r \ll 1$ very small and writing
    \begin{equation*}
        \int_{M\setminus B_R(p)} H_M(x,p,t) \, d\mu(x)  \le \int_{M\setminus B_r(p)} H_M(x,p,t) \, d\mu(x) 
    \end{equation*}
    allows to bound the desired integral.
\end{proof}

Now we present the proof of Lemma \ref{porcatroia}, that we needed to prove the asymptotics of the full $H^{s/2}(M)$ seminorm of Theorem \ref{globlimnoncompact}.

\begin{proof}[Proof of Lemma \ref{porcatroia}]
  Let $\varphi^{-1}:B_1(p) \to \R^n $ be the inverse of the exponential map at $p$. Take $\eta \in C^\infty_c(\B_{4/5}(0))$ with $\chi_{\B_{2/5}(0)} \le \eta \le \chi_{\B_{4/5}(0)}$ and let $g'_{ij}:=g_{ij} \eta +(1-\eta)\delta_{ij}$. This is a metric on $\R^n$ with $g'_{ij}=g_{ij}$ in $\B_{2/5}(0)$. Denote by $\mathcal{K}_s, \mathcal{K}_s'$ the singular kernels of $(M,g)$ and $M':=(\R^n, g')$ respectively. Let $\Lambda := \sup_{x \in B_{1/5}(p)} H_M(x,x,1)$ and $\Lambda' := \sup_{x \in \B_{1/5}(0)} H_{M'}(x,x,1)$. Then, by \cite[Lemma 2.17]{CFS23} applied to the Riemannian manifolds $(M,g)$ and $(\R^n, g')$ we have, for $x,y \in \B_{1/5}(0)$
\begin{align*}
    \big| \mathcal{K}_s(\varphi(x), \varphi(y)) & -\mathcal{K}_s'(x,y) \big|  \le \frac{s/2}{\Gamma(1-s/2)} \int_{0}^\infty \big| H_M(\varphi(x), \varphi(y),t)-H_{M'}(x,y,t) \big| \frac{dt}{t^{1+s/2}} \\ & \le Cs(2-s) \int_{0}^1 \big| H_M(\varphi(x), \varphi(y),t)-H_{M'}(x,y,t) \big| \frac{dt}{t^{1+s/2}} \\ & + Cs(2-s) \int_{1}^{1/s} \big| H_M(\varphi(x), \varphi(y),t)-H_{M'}(x,y,t) \big| \frac{dt}{t^{1+s/2}} \\ & + Cs(2-s) \int_{1/s}^\infty \big| H_M(\varphi(x), \varphi(y),t)-H_{M'}(x,y,t) \big| \frac{dt}{t^{1+s/2}} \\ & := Cs(2-s)\big[I_1+I_2+I_3 \big] \,.
\end{align*}

By \cite[Lemma 2.17]{CFS23} there holds
\begin{equation*}
    I_1 = \int_{0}^1 \big| H_M(\varphi(x), \varphi(y),t)-H_{M'}(x,y,t) \big| \frac{dt}{t^{1+s/2}} \le C \int_{0}^1 e^{-c/t}  \frac{dt}{t^{1+s/2}} \le C \, ,
\end{equation*}
for some dimensional $C=C(n)>0$. Regarding the second integral
\begin{equation*}
    I_2 \le \int_{1}^{1/s} (\Lambda+\Lambda') \frac{dt}{t^{1+s/2}} = (\Lambda+\Lambda') \frac{1-s^{s/2}}{s/2} \, ,
\end{equation*}
and lastly
\begin{align*}
    I_3 & = \int_{1/s}^\infty \big| H_M(\varphi(x), \varphi(y),t)-H_{M'}(x,y,t) \big| \frac{dt}{t^{1+s/2}} \\ & \le  s^{s/2} \int_{1}^\infty \bigg[ H_M(\varphi(x), \varphi(y), \xi/s)+H_{M'}(x,y,\xi/s) \bigg] \frac{d\xi}{\xi^{1+s/2}} = o_s(1) \to 0
\end{align*}
as $s\to 0^+$, since both $M$ and $M'$ have infinite volume, and thus, their heat kernel tends to zero as $t\to + \infty$ (see Lemma \ref{gryg}). Hence as $s\to 0^+$
\begin{equation*}
    \big| \mathcal{K}_s(\varphi(x), \varphi(y))  - \mathcal{K}_s'(x,y) \big| \le Cs+C(\Lambda+\Lambda')(1-s^{s/2}) +o_s(1) = o_s(1) \,,
\end{equation*}
and note that this estimate is uniform in $x,y \in \B_{1/5}(0)$. This follows, for example, from the parabolic Harnack inequality since one can locally estimate the supremum of $H_M$ and $H_{M'}$ with the $L^1$ norm at later times; see the end of the proof of Lemma \ref{gryg}. Then 
 \begin{equation*}
        \lim_{s\to 0^+} \sup_{x,y \in B_{1/8}(p) } \left| \mathcal{K}_s(x,y) - \mathcal{K}_s'(x,y) \right|  =0 \,.
    \end{equation*}
Lastly, by \cite[Lemma 2.5]{CFS23} there exists dimensional constants $c, C >0$ such that 
\begin{equation*}
   c \frac{\beta_{n,s}}{d(x,y)^{n+s}} \le \mathcal{K}_s'(x,y) \le  C \frac{\beta_{n,s}}{d(x,y)^{n+s}} \,,
\end{equation*}
and this concludes the proof.
\end{proof}

\subsection{On the equivalence and well-posedness of different fractional Laplacians.}\label{flapsection}

In this subsection we shall prove some results concerning the equivalence between different definitions of the fractional Laplacian, and the fractional Sobolev seminorms on (possibly weighted) Riemannian manifolds. 

\vsp
Next we want to show that the fractional laplacian defined with the heat semigroup $(-\Delta)^{s/2}_{\rm B} $ and the one defined via the singular integral $(-\Delta)^{s/2}_{\rm Si} $ coincide. Note that the two following propositions do not hold when $M$ is not stochastically complete. Indeed, using definition \eqref{flap} gives $(-\Delta)^{s/2}_{\rm Si} (1)  \equiv 0$, while if $M$ is not stochastically complete equation \eqref{boclap} gives $(-\Delta)^{s/2}_{\rm B}(1) \neq 0$.

    \begin{proposition}\label{coincideprop3}
        Let $(M,g)$ be a complete, stochastically complete Riemannian manifold, and let $u\in C^\infty_c(M)$. Then:
        \begin{itemize}
            \item[(i)] For $s<1$ the integral in $(-\Delta)^{s/2}_{\rm Si} u$ is absolutely convergent and the principal value is not needed. 
            \item[(ii)] The singular integral $(-\Delta)^{s/2}_{\rm Si} u$ (defined in \eqref{flap}) and the Bochner $(-\Delta)^{s/2}_{\rm B} u$ (defined in \eqref{boclap}) definition coincide.
        \end{itemize}
    \end{proposition}
\begin{proof} For what concerns the absolute convergence for $s\in (0,1)$, we have
   \begin{align*}
       \int_M(u(x)-u(y))\mathcal{K}_s(x,y)d\mu(y) &=   \int_{B_r(x)} (\dotsc) \, d\mu(y) +   \int_{M\setminus B_r(x)}(\dotsc) \, d\mu(y)  =: I_1+I_2.
   \end{align*}
   For $r$ small, arguing exactly as in the proof of Theorem \ref{globlimnoncompact}
   \begin{equation*}
       I_1\leq C\int_{B_r(x)}\frac{1}{d(x,y)^{n+s-1}}d\mu(y)\leq C\int_0^{r}\frac{1}{\rho^{s}}d\rho<+\infty \,.
   \end{equation*}
   On the other hand, for the second integral
   \begin{equation*}
       I_2\leq 2\|u\|_{L^\infty} \int_{M\setminus B_r(x)}\mathcal{K}_s(x,y)d\mu(y) \,,
   \end{equation*}
   and thanks to Lemma \ref{convergence} and Fubini
   \begin{align*}
       \int_{M\setminus B_r(x)}\mathcal{K}_s(x,y)d\mu(y) &= \int_0^\infty\frac{1}{t^{1+s/2}}\int_{M\setminus B_r(x)}H_M(x,y,t)d\mu(y)dt \\ & \leq C \int_0^1e^{-c/t}\frac{dt}{t^{1+s/2}}+\int_1^\infty\frac{1}{t^{1+s/2}}dt<+\infty.
   \end{align*}
   This concludes the proof of $(i)$.

   \vsp
   Now, let us define
    \begin{equation*}
        \mathfrak{J}(t) :=\frac{e^{t\Delta}u(x)-u(x)}{t^{1+s/2}}=\frac{1}{t^{1+s/2}}\int_MH_M(x,y,t)(u(y)-u(x))d\mu(y),
    \end{equation*}
    where the second equality is due to the stochastical completeness. Note that $ \mathfrak{J} \in L^1(0, +\infty)$ since $|e^{t\Delta}u(x)-u(x)|\leq Ct$, where the constant $C$ depends on $\|\Delta u\|_{L^\infty}$. We can now define
    \begin{equation*}
          \mathfrak{J}_k(t) :=\frac{1}{t^{1+s/2}}\int_{M\setminus B_{1/k}(x)}H_M(x,y,t)(u(y)-u(x))d\mu(y).
    \end{equation*}
and observe that $ \mathfrak{J}_k(t) \to  \mathfrak{J}(t)$ for all $t\in(0,\infty)$. Now if $t\geq 1$ (estimating the mass of the heat kernel by $1$) we get $\mathfrak{J}_k(t)\leq 2\|u\|_{L^\infty}/t^{1+s/2}$, while by \cite[Lemma 2.11]{CFS23} we have 
\begin{align*}
        \big|\mathfrak{J}(t)-\mathfrak{J}_k(t) \big| & \le  \frac{1}{t^{1+s/2}}\int_{B_{1/k}(x)}H_M(x,y,t)|u(y)-u(x)|d\mu(y) \\ & \le \frac{C}{t^{1+s/2+n/2}}\int_{B_{1/k}(x)}e^{-d^2(x,y)/5t}d(x,y)d\mu(y) \,.
    \end{align*}
Applying Coarea formula and using the fact that ${\rm Per}(B_r(x))\leq C r^{n-1}$ if $k$ is big we get 
    \begin{equation*}
         \big| \mathfrak{J}(t)-\mathfrak{J}_k(t) \big| \leq\frac{C}{t^{1+s/2+n/s}}\int_0^{1/k}e^{-r^2/5t}r^ndr = \frac{C}{t^{s/2}}\int_0^{1/(5tk^2)}e^{-z}z^{n/2-1}dz\leq\frac{C}{t^{s/2}}.
    \end{equation*}
Therefore if $t\geq 1$ we have $ \mathfrak{J}_{k}(t)\leq C/t^{1+s/2} \in L^1(1, +\infty)$ while if $t\leq 1$ we have $\mathfrak{J}_{k}(t)\leq C/t^{s/2}+\mathfrak{J}(t) \in L^1(0,1)$. Hence by dominated convergence we can write
\begin{align*}
    (-\Delta)^{s/2}_{\rm B} u (x)  =\int_0^\infty \mathfrak{J}(t) \, dt = \lim_{k \to \infty} \int_{0}^\infty \int_{M\setminus B_{1/k}(x)} (u(y)-u(x))H_M(x,y,t) \frac{dt}{t^{1+s/2}} \, d\mu(y).
\end{align*}
Now for any $k\in\mathbb{N}$ fixed, by Lemma \eqref{convergence} and the fact that $u$ is bounded, we get  
\begin{equation*}
    \int_0^\infty\int_{M\setminus B_{1/k} (x)} |u(y)-u(x)|H_M(x,y,t) \frac{dt}{t^{1+s/2}}\leq 2\|u\|_{L^\infty} \int_0^1 e^{-c/t} \frac{dt}{t^{1+s/2}}+ 2\|u\|_{L^\infty}\int_1^\infty \frac{dt}{t^{1+s/2}}<+\infty.
\end{equation*}
Therefore we can apply Fubini and infer 
\begin{align*}
     (-\Delta)^{s/2}_{\rm B} u (x) & = \lim_{k \to \infty} \int_{M\setminus B_{1/k}(x)} \int_{0}^\infty (u(y)-u(x))H_M(x,y,t) \frac{dt}{t^{1+s/2}} \, d\mu(y) \\ & ={\rm P.V.}\int_M(u(y)-u(x))\mathcal{K}_s(x,y)d\mu(y).
\end{align*}
\end{proof}
\begin{remark}
   One can note that the proof above of the absolute convergence of $(-\Delta)^{s/2}_{\rm Si} u $ for $s\in (0,1)$ actually shows that the integral is absolutely convergent if $u \in C^{\alpha}_{\rm loc}(M) \cap L^\infty(M)$ for some $\alpha >s$.  
\end{remark}

Regarding the following two results, we couldn't find any proof in the case of an ambient Riemannian manifold $(M,g)$, even though they appear to be well-known in the community in the case $M=\mathbb{R}^n$ or a domain $M=\Omega \subset \R^n$. For example, a proof that ${\rm Dom}((-\Delta_{\Omega})^{s/2}_{\rm Spec}) = H^{s}(\Omega)$ for the Dirichlet Laplacian on $\Omega \subset \R^n$ can be found in \cite[Section 3.1.3]{BonforteetAl}, but it heavily uses the discreteness of the spectrum and interpolation theory. 

\vsp
Our results are not sharp, in particular, we believe that Proposition \ref{coincideprop2} and \ref{coincideprop4} hold also for $s=\sigma$ since this is the case for domains in $\R^n$. Here we focus on providing structural (and short) proofs that apply verbatim to the case of any weighted manifold, and we avoid using any local Euclidean-like structure of $M$. 

\begin{proposition}\label{coincideprop2}
    Let $(M,g)$ be a stochastically complete Riemannian manifold, $\sigma \in (0,1)$ and $u \in H^\sigma(M) $ (as defined in Definition \ref{fracsobdef}). Then, for every $s <\sigma$ the singular integral $(-\Delta)^{s/2}_{\rm Si} u$ (defined in \eqref{flap}) and the Bochner $(-\Delta)^{s/2}_{\rm B} u$ (defined in \eqref{boclap}) definition coincide a.e. Moreover $(-\Delta)^{s/2}_{\rm B} u = (-\Delta)^{s/2}_{\rm Si} u \in L^2(M)$.  
    \end{proposition}
    \begin{proof}
        Let $u \in H^\sigma(M)$ and $x\in M$. Since $M$ is stochastically complete, if we could exchange the order of integration we would have
        \begin{align*}
             (-\Delta)^{s/2}_{\rm B} u (x)&= \frac{1}{\Gamma(-s/2)} \int_{0}^{\infty} (e^{t\Delta}u (x)-u(x))\frac{dt}{t^{1+s/2}} \\ & = \frac{1}{\Gamma(-s/2)} \int_{0}^{\infty} \left( \int_M H_M(x,y,t) (u(y)-u(x)) d\mu(y) \right)\frac{dt}{t^{1+s/2}} \\ &= \int_M (u(y)-u(x))\mathcal{K}_s(x,y) d\mu(y) = (-\Delta)_{\rm Si}^{s/2} u (x) \,.
        \end{align*}
        Now we shall justify the steps above, showing that the integral is absolutely convergent. Note that this will also justify the last equality, since we have defined $(-\Delta)_{\rm Si}^{s/2}$ with the Cauchy principal value. In particular, we show that 
\begin{equation*}
    \int_M \left( \int_M |u(x)-u(y)| \mathcal{K}_s(x,y) d\mu(y)\right)^2 d\mu(x) <+\infty \,.
\end{equation*}
This will prove at the same time that the integral above is absolutely convergent for a.e. $x\in M$ and that $(-\Delta)^{s/2}_{\rm Si} u \in L^2(M)$. 
Let us call
        \begin{equation*}
            \mathfrak{I}(t):=\int_{M} |u(x)-u(y)| H_M(x,y,t) \, d\mu(y) \, ,
        \end{equation*}
and denote by $C$ a constant that depends at most on $\sigma$. 

\vsp
Note that, by Jensen's inequality 
\begin{align}
           \int_0^\infty \mathfrak{I}(t)^2 \frac{dt}{t^{1+\sigma}} &= \int_0^\infty \left( \int_{M} |u(x)-u(y)| H_M(x,y,t) \, d\mu(y) \right)^2\frac{dt}{t^{1+\sigma}} \nonumber \\ &\le \int_0^\infty\int_M |u(x)-u(y)|^2 H_M(x,y,t) \, d\mu(y)\frac{dt}{t^{1+\sigma}} \nonumber \\ & = C \int_M |u(x)-u(y)|^2 \mathcal{K}_{2\sigma}(x,y) \, d\mu(y) \,. \label{asfgae}
 \end{align}
Write
\begin{align*}
      \int_M  \bigg( \int_M |u(x)-u(y)| & \mathcal{K}_s(x,y) d\mu(y)\bigg)^2 d\mu(x)  \\ & = Cs^2\int_M\left(\int_0^\infty \mathfrak{I}(t)\frac{dt}{t^{1+s/2}} \right)^2d\mu  \\ &  \le Cs^2 \int_M   \left( \int_0^1\mathfrak{I}(t)\frac{dt}{t^{1+s/2}}\right)^2d\mu + Cs^2\int_M \left( \int_1^\infty \mathfrak{I}(t)\frac{dt}{t^{1+s/2}} \right) ^2d\mu  \,.
\end{align*}
For the first integral, since $s<\sigma$, by H\"older's inequality and \eqref{asfgae} we have
\begin{equation*}
           \int_M \left( \int_0^1 \mathfrak{I}(t)\frac{dt}{t^{1+s/2}} \right)^2 d\mu \le \int_M \left(\int_0^1 \mathfrak{I}(t)^2 \frac{dt}{t^{1+\sigma}} \right) \left( \int_0^1 \frac{dt}{t^{1-\sigma+s}}\right) d\mu \le C [u]_{H^\sigma(M)}^2 <+\infty.
\end{equation*}
For the second integral, let us first renormalize the measure $\nu := Cdt/t^{1+s/2}$ in a way that it becomes a probability measure on $[1,\infty)$. Then, by Jensen again (applied two times: to $d\nu(t)$ and then $H_M(x,y,t)d\mu(y)$) 
\begin{align*}
    \int_M \left( \int_1^\infty \mathfrak{I}(t)\frac{dt}{t^{1+s/2}} \right)^2 d\mu & \leq \frac{C}{s^2}\iint_{M\times M}\int_1^\infty |u(x)-u(y)|^2 H_M(x,y,t) \, d\nu(t) d\mu(y)d\mu(x) \\
    &\leq \frac{4C}{s^2}\iint_{M\times M}\int_1^\infty |u(x)|^2 H_M(x,y,t) \, d\nu(t) d\mu(y)d\mu(x) \\ &\leq \frac{4C}{s^2}\|u\|_{L^2(M)}^2 <+\infty \,.
\end{align*}
Hence, we have proved
\begin{align}
    \| (-\Delta)^{s/2}_{\rm Si} u \|_{L^2(M)}^2 & \le \int_M \left( \int_M |u(x)-u(y)| \mathcal{K}_s(x,y) d\mu(y)\right)^2 d\mu(x) \nonumber \\ & \le C \|u\|_{L^2(M)}^2 +  Cs^2\| u \|_{H^\sigma(M)}^2 , \label{uqweriufgweur}
\end{align}
and this concludes the proof.
\end{proof}
\begin{remark}\label{rcdrmk}
    Note that the proof of Proposition \ref{coincideprop2} applies verbatim to the case of ${\rm RCD}(K,N)$ spaces, since every ${\rm RCD}(K,N)$ space is stochastically complete. We will use this fact in the proof of Theorem \ref{rcdlimit22}. 
\end{remark}

Next, we address the equivalence of the spectral fractional Laplacian $(-\Delta)^{s/2}_{\rm Spec}$ with the other definitions. We refer to \cite{GrygBook} and \cite[Section 2.6]{Spec1} and the references therein for an introduction of the spectral theory of the fractional Laplacian on general spaces.

     \vsp
    Let ${E_\lambda}$ be the spectral resolvent of (minus) the Laplacian on $(M,g)$. Then, for $s\in (0,2)$ in the classical sense of spectral theory 
\begin{equation*}
    {\rm Dom}((-\Delta)^{s/2}_{\rm Spec}) := \Big\{u \in L^2(M) \, : \, \int_{\sigma(-\Delta)} \lambda^{s} \, d\lp E_\lambda u, u \rp <+\infty \Big\} \,,
\end{equation*}
and for $u \in {\rm Dom}((-\Delta)^{s/2}_{\rm Spec})$ 
\begin{equation}\label{speclap}
    (-\Delta)^{s/2}_{\rm Spec} u := \int_{\sigma(-\Delta)} \lambda^{s/2} d \lp E_\lambda u , \cdot  \rp \,.
\end{equation}

\begin{proposition}\label{coincideprop4}
    Let $(M,g)$ be a stochastically complete Riemannian manifold, $\sigma \in (0,1)$ and $s<\sigma$. Then $ H^\sigma(M)  \subseteq {\rm Dom}((-\Delta)^{s/2}_{\rm Spec})$.  
\end{proposition}
\begin{proof}
    Let $u \in H^\sigma(M)$, and let 
    \begin{equation*}
        \varphi(\lambda) := \lambda^{s/2} =  \frac{1}{\Gamma(-s/2)} \int_0^\infty
(e^{-\lambda t}-1)\frac{dt}{t^{1+s/2}} \,.
    \end{equation*}
    Since $u\in L^2(M)$, by standard spectral theory (see \cite{GrygBook} for example)
    \begin{align*}
       \int_0^\infty \lambda^s d\lp E_\lambda u, u \rp & = \int_0^\infty |\varphi(\lambda)|^2 d\lp E_\lambda u, u \rp = \|\varphi(-\Delta) u \|_{L^2(M)}^2 \\ &= \left\| \int_0^\infty (e^{t\Delta}u-u) \frac{dt}{t^{1+s/2}} \right\|_{L^2(M)}^2 = \| (-\Delta)^{s/2}_{\rm B} u \|_{L^2(M)}^2 =  \| (-\Delta)^{s/2}_{\rm Si} u \|_{L^2(M)}^2 <+\infty \,,
    \end{align*}
    where we have used that by Proposition \ref{coincideprop2} $(-\Delta)^{s/2}_{\rm B} u = (-\Delta)^{s/2}_{\rm Si} u \in L^2(M)$. 
\end{proof}

\begin{proposition}\label{coincideprop1}
    Let $u \in {\rm Dom}((-\Delta)_{\rm Spec}^{s/2}) $. Then 
    \begin{equation*}
       (-\Delta)_{\rm B}^{s/2} u :=  \frac{1}{\Gamma(-s/2)} \int_{0}^{\infty} (e^{t\Delta}u-u)\frac{dt}{t^{1+s/2}} = \int_{\sigma(-\Delta)} \lambda^{s/2} d \lp E_\lambda u , \cdot  \rp =: (-\Delta)_{\rm Spec}^{s/2} u \,,  
    \end{equation*}
    where the equality is in duality with ${\rm Dom}((-\Delta)_{\rm Spec}^{s/2})$. 
    
\end{proposition}
\begin{proof}
    We follow \cite[Lemma 2.2]{CafStinga} which deals with the analogous proposition in the case of discrete spectrum in a domain $\Omega \subset \R^n$. Recall the numerical formula
    \begin{equation*}
        \lambda^{s/2} = \frac{1}{\Gamma(-s/2)} \int_0^\infty
(e^{-\lambda t}-1)\frac{dt}{t^{1+s/2}} \,,   
\end{equation*}
valid for $\lambda>0$, $0<s<2$. Let $\psi \in {\rm Dom}((-\Delta)_{\rm Spec}^{s/2})$, and write $\psi = \int_{\sigma(-\Delta)} dE_\lambda \lp \psi, \cdot \rp $. Then 
    \begin{align*}
        \int_{\sigma(-\Delta)} \lambda^{s/2} d \lp E_\lambda u , \psi \rp &= \frac{1}{\Gamma(-s/2)} \int_{\sigma(-\Delta)} \int_0^\infty (e^{-\lambda t}-1)\frac{dt}{t^{1+s/2}} d \lp E_\lambda u , \psi \rp \\ &= \frac{1}{\Gamma(-s/2)}  \int_0^\infty \left( \int_{\sigma(-\Delta)} (e^{-\lambda t}-1) d \lp E_\lambda u , \psi \rp \right)  \frac{dt}{t^{1+s/2}}  \\ &= \frac{1}{\Gamma(-s/2)}  \int_0^\infty \big(\lp e^{t \Delta } u , \psi\rp -\lp u, \psi \rp \big)   \frac{dt}{t^{1+s/2}} \,,
    \end{align*}
    where the second-last inequality follows by Fubini's theorem since $u,\psi \in {\rm Dom}((-\Delta)_{\rm Spec}^{s/2}) $. 
\end{proof}
\begin{corollary}\label{corolequivalence}
    Let $(M,g)$ be a stochastically complete Riemannian manifold,  $\sigma\in (0,1)$, $s<\sigma$ and $u \in H^\sigma(M)$. Then 
    \begin{equation*}
     \frac{1}{2} [u]_{H^{s/2}(M)}^2 =  \int_{M} u (-\Delta)^{s/2}_{\rm Si} u \, d\mu = \int_0^\infty \lambda^{s/2} d\lp E_\lambda u,u \rp \,.
\end{equation*}
\end{corollary}
\begin{proof}
   The first equality is \eqref{vbnvbn}, and the second equality is a direct consequence of Proposition \ref{coincideprop2}, Proposition \ref{coincideprop4} and Proposition \ref{coincideprop1}.
\end{proof}

\subsection{Manifolds with nonnegative Ricci curvature.}

We recall a theorem of Yau which gives a lower bound on the growth of the volume of geodesic balls under the nonnegative Ricci curvature assumption. Note that the same holds with the same proof on ${\rm CD}(K,N)$ spaces.
\begin{theorem} 
\label{Yau volume growth}
Let $(M,g)$ be a complete non-compact Riemannian manifold with $\Ric_M \ge 0 $. Then, there exists a constant $C=C(n)>0$ such that for every $x \in M$ and $\lambda>0$
\begin{equation*}
    V_x(r \lambda) \ge C r V_x(\lambda), \, \quad \forall \, r>1 \,.
\end{equation*}
\end{theorem}
\begin{proof}
    By scaling invariance of the hypothesis $\Ric_M\ge 0$ one can assume $\lambda=1$. Then, the result is \cite[Theorem 2.5]{Liga}. 
\end{proof}

Next, we present here a result concerning the growth of the singular kernel $\mathcal{K}_s$ in the case of nonnegative Ricci curvature. We will not use this result anywhere but we believe it can be interesting per se. For example, it implies that on cylinders $M=\Sp^{n-k} \times \R^{k}$ (with their product metric) the singular kernel $\mathcal{K}_{s}(x,y)$ decays like $1/d(x,y)^{k+s}$ and not $1/d(x,y)^{n+s}$ for large distances.

\begin{lemma}
    Let $(M,g)$ be an $n$-dimensional Riemannian manifold with ${\rm Ric}_M\geq 0$ and $s \in (0,2)$. Then, there exists dimensional constants $0<c<C$ such that 
    \begin{equation*}
        c \frac{s(2-s)}{r^{s}\mu(B_r(x))}\leq \mathcal{K}_s(x,y)\leq C \frac{s(2-s)}{r^{s}\mu(B_r(x))}
    \end{equation*}
    with $r=d(x,y)$ for all $x,y\in M$.
\end{lemma}
\begin{proof}
    In the definition of the singular kernel $\mathcal{K}_s$ we first perform the change of variables $r^2t=k$ with $r=d(x,y)$ so that we obtain 
    \begin{equation*}
        \mathcal{K}_s(x,y)=\frac{r^{-s}}{|\Gamma(-s/2)|} \int_0^\infty H_M(x,y,r^2k)\frac{dk}{k^{1+s/2}}.
    \end{equation*}
    Now we employ the Gaussian estimates from above to get
    \begin{equation*}
        \mathcal{K}_s(x,y)\leq \frac{Cs(2-s)}{r^s}\biggl[\int_0^1\frac{1}{\mu(B_{r\sqrt{k}}(x))}e^{-1/5k}\frac{dk}{k^{1+s/2}}+\int_1^\infty\frac{1}{\mu(B_{r\sqrt{k}}(x))}e^{-1/5k}\frac{dk}{k^{1+s/2}}\biggr]=:I_1+I_2.
    \end{equation*}
    Using Bishop-Gromov's inequality we get 
    \begin{equation*}
        I_1\leq \frac{Cs(2-s)}{\mu(B_r(x))}\int_0^1 \frac{e^{-1/5k}}{k^{n/2+1+s/2}}dk\leq \frac{Cs(2-s)}{\mu(B_r(x))},
    \end{equation*}
    while for $k\in(1,\infty)$ we can use Theorem \ref{Yau volume growth} to write 
    \begin{equation*}
        I_2\leq \frac{Cs(2-s)}{\mu(B_r(x))}\int_1^\infty e^{-1/5k}\frac{dk}{k^{3/2+s/2}}\leq\frac{Cs(2-s)}{\mu(B_r(x))}
    \end{equation*}
    and this concludes the upper estimate. For the one from below we again use the Gaussian estimates to infer
    \begin{equation*}
         \mathcal{K}_s(x,y)\geq \frac{cs(2-s)}{r^s}\biggl[\int_0^1\frac{1}{\mu(B_{r\sqrt{k}}(x))}e^{-1/3k}\frac{dk}{k^{1+s/2}}+\int_1^\infty\frac{1}{\mu(B_{r\sqrt{k}}(x))}e^{-1/3k}\frac{dk}{k^{1+s/2}}\biggr]=:I_3+I_4.
    \end{equation*}
    We now get 
    \begin{equation*}
        I_3\geq \frac{cs(2-s)}{\mu(B_{r}(x))}\int_0^1 e^{-1/3k}\frac{dk}{k^{1+s/2}}=\frac{cs(2-s)}{\mu(B_{r}(x))}.
    \end{equation*}
    Since $I_4\geq 0$ we infer the lower bound as well.
\end{proof}

\begin{remark}
    If we assume ${\rm AVR}(M)=\lim_{r\to\infty}\frac{\mu(B_r(x))}{\omega_n r^n}=\theta>0$ then we have the more Euclidean-like bounds
    \begin{equation*}
       \frac{cs(2-s)}{\theta r^{n+s}} \leq \mathcal{K}_s(x,y)\leq  \frac{Cs(2-s)}{ \theta r^{n+s}}.
    \end{equation*}
    Note moreover that the same proof works in the singular setting of ${\rm RCD}(0,N)$ spaces.
\end{remark}

	 	 	\bibliography{references}
            \bibliographystyle{abbrv}
    % \printbibliography
	 	 	
\end{document}